\documentclass[11pt,reqno]{amsart}
\usepackage[margin= 2.0cm]{geometry}
\usepackage{algorithm}
\usepackage{algpseudocode}
\usepackage{tikz}

\usepackage[textsize=small]{todonotes}
\usepackage{amsmath}
\usepackage{amsfonts}
\usepackage{amssymb}
\usepackage{amsthm}
\usepackage{empheq}    

\usepackage{mathtools} 
\usepackage{latexsym}

\usepackage{enumerate}
\usepackage{cancel}
\usepackage{ragged2e}
\usepackage{blindtext}
\usepackage{cases}
\usepackage{empheq}
\usepackage{multicol}
\usepackage{placeins} 

\usepackage{wrapfig}
\usepackage{caption}

\usepackage{subfigure}
\usepackage{float}

\usepackage{color}

\usepackage{mathrsfs}
\usepackage{fontenc} 
\usepackage{inputenc} 
\usepackage{enumitem}
\usepackage{verbatim}

\theoremstyle{plain}
\newtheorem{theorem}{Theorem}[section]

\newtheorem{lemma}[theorem]{Lemma}

\theoremstyle{definition}
\newtheorem{definition}[theorem]{Definition}

\theoremstyle{remark}
\newtheorem{remark}[theorem]{Remark}
\usepackage[square,comma,numbers,sort]{natbib}
\usepackage[colorlinks=true, pdfborder={ 0 0 0}]{hyperref}
\hypersetup{urlcolor=blue, citecolor=red}
\usepackage{url}
\makeatletter
\renewcommand{\section}{\@startsection{section}{1}{0pt}%
  {1.5ex plus .2ex minus .2ex}%
  {1.0ex plus .2ex}%
  {\normalfont\large\bfseries\raggedright}}
\renewcommand{\subsection}{\@startsection{subsection}{2}{0pt}%
  {1.25ex plus .2ex minus .2ex}%
  {1.0ex plus .2ex}%
  {\normalfont\normalsize\bfseries\raggedright}}
\makeatother

\allowdisplaybreaks
\makeatletter
\def\abstractname{Abstract}
\renewenvironment{abstract}{%
  \begin{center}
    \textbf{\abstractname}
  \end{center}
}{}
\makeatother

\usepackage{xcolor}
\begin{document}
\title[PINNs in FSI model]{Error estimates in Physics-Informed Neural Networks to heat-wave Fluid-Structure Interactive system}
\author[Geredeli]{Pelin G. Geredeli}
\address{School of Mathematical and Statistical Sciences, Clemson University, Clemson, SC 29634, USA}
\email{pgerede@clemson.edu}
\author[Rahman]{Mohammad Mahabubur Rahman}
\address{School of Mathematical and Statistical Sciences, Clemson University, Clemson, SC 29634, USA}
\email{rahman6@clemson.edu}
\author[Verma]{Deepanshu Verma}
\address{School of Mathematical and Statistical Sciences, Clemson University, Clemson, SC 29634, USA}
\email{dverma@clemson.edu}
\maketitle{}
\begin{abstract}
We study a physics-informed neural network (PINN) approximation of a coupled PDE system that models the interaction between fluid and elastic dynamics. The canonical heat-wave model considered here couples a parabolic and a hyperbolic equation through transmission conditions on a fixed boundary interface, and serves as a prototype for the Stokes-elastic systems that arises in real world applications. This coupling makes the error analysis more delicate than that for a single PDE. We bound the error between the exact and PINN solutions by the associated PINN residuals, and we moreover prove that the generalization error is controlled by the training error, the number of quadrature points, and the size of the network. {To the best of our knowledge}, these are the first error estimates for a PINN approximation of a coupled fluid-structure interaction PDE system. We also show that numerical experiments with respect to a given test problem.

\end{abstract}
\bigskip

\medskip


\bigskip

\section{Introduction}
The fluid structure interaction (FSI) PDE models which are composed of elastic dynamics coupled to heat, Stokes or Navier Stokes flows are known to describe a variety of physical phenomena, e.g., vibrations in elastic membranes, aerodynamics of structures such as bridges and tall buildings \cite{AT,AG2,Barbu,Bociu,Gunz,P8,P9,P10}. In particular, the mathematical analysis of  FSI problems in which the governing equations share the same dimensionality; i.e., 2D or 3D structure equation coupled to 2D or 3D fluid flow, and the boundary interface is particularly \textit{\underline{not}} smooth becomes quite challenging \cite{AG2, duc, zua1, zua2}.

Analytical results for various linear and nonlinear 2D (3D)-2D (3D) FSI models, such as those of well-posedness and longtime behavior of solutions, have been studied in various settings; see \cite{AT,AR,AG,AG2,Barbu,BC}. In particular, the semigroup wellposedness of the Stokes-Lame interaction coupled PDE dynamics  was studied in \cite{AT,AR}. Subsequently, in \cite{AG}, the authors proved that the solution to the PDE system considered in \cite{AR} satisfies higher regularity if the initial data is taken to be sufficiently smooth. The stability analysis of the solution to that Stokes-Lame PDE system was also studied in \cite{AR} and \cite{AG2} in the sense of strong and rational decay. 



The numerical approximation of coupled FSI systems has also been a research objective; see, e.g., \cite{bur,P8, P10, ku} 
and references therein. Although finite element methods have often been invoked for coupled PDE systems, they require meshes that are compatible across the fluid and structure subdomains, which is often difficult to ensure for complex interfaces. PINNs, which enforce the equations and the transmission conditions through a single mesh-free residual loss, have recently emerged as a widely used framework for approximating solutions of PDEs across a broad range of applications in science and engineering. While the 
foundational ideas for PINNs date back to the 1990s \cite{15a,40l,41p}, there 
has been a notable resurgence of interest in recent years 
\cite{RAM,RA,14p,KA21}. In the PINN approach, a neural network is 
``trained" by minimizing residuals of the governing equations at 
collocation points in the interior, on the boundary, and at the 
initial time, incorporating the underlying physical laws directly 
into the training process. For a comprehensive overview of recent 
developments, we refer the reader to \cite{14p,KA21} and 
references therein.

Deriving rigorous error estimates for PINNs has drawn considerable 
attention from the mathematical community, leading to substantial 
progress for a number of standard PDE settings 
\cite{DE,DE1,MS,MS1}. In \cite{MS1}, Mishra and Molinaro established 
an abstract framework in which the stability properties of the 
underlying PDE are leveraged to bound the generalization error of a 
trained PINN, in terms of the training error and the number of 
collocation points. Since then, this approach has been extended to 
several nonlinear PDEs including the Navier--Stokes and 
Kuramoto--Sivashinsky equations \cite{DE,DE1,rahman_verma}. However, despite this progress, the application of PINNs to \emph{coupled multi-physics} systems together with a rigorous mathematical justification remains limited; see \cite{F}, where PINNs were applied to a two-dimensional Navier-Stokes fluid flow interacting with an elastic disc and this numerical method was compared against a classical solver, with no error analysis. Error estimates for related settings exist for thermally coupled Navier--Stokes flows \cite{Goraya2023}, single parabolic interface problem \cite{kumar2025}, and domain decomposition settings \cite{hu2022}; but there is nothing in the literature for the treatment of a parabolic-hyperbolic system coupled through interface conditions. What these references lack are conditions which guarantee that a small training loss implies a small error in the coupled solution.

To the best of the authors' knowledge, no error estimates 
exist for PINNs which approximate coupled fluid-structure interaction 
PDE systems --even in the canonical case. Motivated by this gap, the present manuscript develops the first rigorous error analysis for PINNs applied to a  
coupled FSI system. We address the following key 
theoretical questions:

\begin{itemize}
\item[\textbf{[Q1]}]
Let $(w,w_t,u)$ denote the exact solution of the coupled system and let
$(w_\theta,w_{\theta t},u_\theta)$ be the corresponding PINN approximation,
where $\theta\in\Theta$ denotes the network parameters (weights and biases).
Given a tolerance $\varepsilon>0$, if the generalization error
$
E_G(\theta)<\varepsilon,
$
does there exist $\delta(\varepsilon)>0$, with
\[
\delta(\varepsilon)\to0
\qquad\text{as}\qquad
\varepsilon\to0,
\]
such that
\[
\left\|
(w,w_t,u)-(w_\theta,w_{\theta t},u_\theta)
\right\|
<
\delta(\varepsilon),
\]
where $\|\cdot\|$ denotes a suitable energy norm?
\item \textbf{[Q2]} Given a small training error $E_T(\theta^*)$ 
and a sufficiently large training set $S$, is the corresponding 
generalization error $E_G(\theta^*)$ also proportionately small?
\end{itemize}


\noindent For the present configuration, the solution is 
approximated by a neural network whose parameters are obtained by 
minimizing residual losses of the PDEs in the fluid and solid 
regions together with the interface and initial conditions 
evaluated at sampled collocation points, thereby yielding a mesh-free 
formulation. We answer the above questions \textbf{[Q1]} and 
\textbf{[Q2]} affirmatively through a rigorous error analysis. 
In particular, we present here a comprehensive treatment of 
PINNs applied to a coupled FSI model which describes the interactive dynamics of canonical heat and wave equations. Unlike single equation PDE problems, 
the heat-wave FSI system involves the interplay between a 
parabolic (heat) component and a hyperbolic (wave) component 
coupled at a (fixed) interface, which introduces distinctive challenges 
in both the energy analysis and treatment of interface 
conditions. In addition, we provide rigorous error bounds for the proposed 
PINN approximation algorithm, and numerical 
experiments in order to validate the accuracy of our result through comparisons with  solutions of a given test problem and the theoretical convergence bounds. The PINN analysis for the more complicated FSI model governed by 
Navier-Stokes fluid and elastic wave/plate equations will be 
treated in a forthcoming publication.
 
\textbf{Plan of the Paper.} Section 1 is devoted to the recent literature, and the introduction of the coupled FSI PDE system under consideration. In Section 2, we give some preliminaries and a functional setup which invokes linear semigroup theory. Also, we present some of the earlier theoretical results related to the wellposedness and higher regularity of the solutions to a coupled PDE system which is the non-canonical version of our coupled PDE dynamics. In Section 3, we present the PINNs formulation for the heat-wave FSI PDE system and derive the corresponding theoretical error bounds. Section 4 is devoted to numerical experiments that validate the given error estimates. In Section 5, we give Appendix.

\subsection{Description of the Coupled Heat-Wave Interaction PDE Problem}
In this paper, we study a canonical heat-structure interaction PDE model which serves as a foundational prototype for the coupled PDE systems which govern the Stokes fluid and elastic equations. The coupled system under consideration describes the interactive dynamics between the fluid (heat-conducting) and elastic structure components through the interaction interface. The FSI geometry is described as follows: Let $\Omega_f \subset \mathbb{R}^d$, with $d = 2$ or $3$, denote the bounded domain where the heat component evolves. Its boundary is partitioned as $\partial \Omega_f = \Gamma_s \cup \Gamma_f$, with $\Gamma_s \cap \Gamma_f = \emptyset$, and each portion is assumed to be sufficiently smooth. The structural domain $\Omega_s$, completely embedded within $\Omega_f$, represents the domain where the elastic component evolves over time. The coupling between the heat and structural dynamics occurs along the interface $\Gamma_s = \partial \Omega_s$ which is assumed to be fixed. 
\begin{figure}
    \centering
    \includegraphics[width=0.5\linewidth]{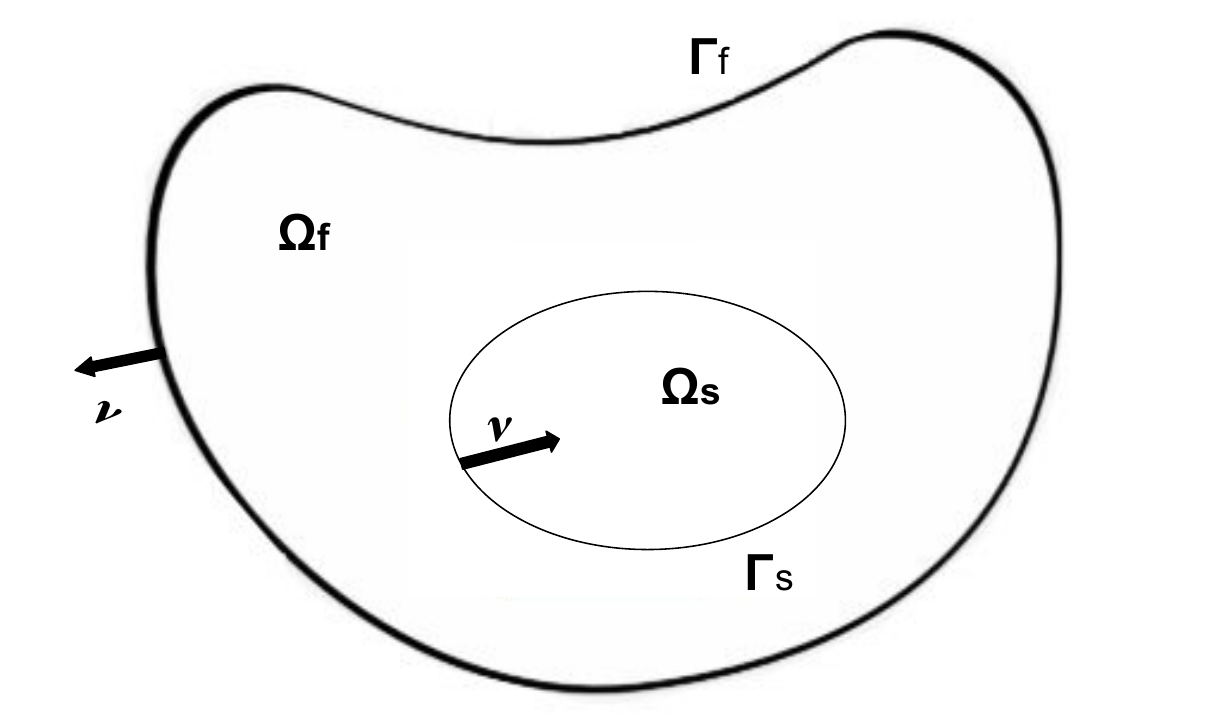}
    \caption{Geometry of the FSI Domain}
    \label{fig:geometry}
\end{figure}

\noindent As shown in Figure~\ref{fig:geometry}, $\nu(x)$ is the unit normal vector outward on  $\Gamma_f,$ and inward on $\Gamma_s.$ The coupled PDE system under consideration is given in solution variables 
$
u = [u_1(t,x), u_2(t,x), \dots, u_d(t,x)],
$
(fluid flow),
$
w = [w_1(t,x), w_2(t,x), \dots, w_d(t,x)],
$
(structure displacement), and
$
w_t = [w_{1t}(t,x), w_{2t}(t,x), \dots, w_{dt}(t,x)],
$
(structure velocity) as:

\begin{equation}
\label{1.0}
\left\{
\begin{aligned}
u_t - \Delta u &= f_1 &&\text{in } (0,T)\times\Omega_f, \\[2mm]
w_{tt} - \Delta w + w &= f_2 &&\text{in } (0,T)\times\Omega_s, \\[1mm]
u &= 0 &&\text{on } (0,T)\times\Gamma_f, \\[1mm]
u &= w_t &&\text{on } (0,T)\times\Gamma_s, \\[1mm]
\partial_\nu u &= \partial_\nu w &&\text{on } (0,T)\times\Gamma_s, \\[1mm]
[w(0),\, w_t(0),\, u(0)] &= [w_0,\, w_{01},\, u_0] &&\text{in } \mathbf{H},
\end{aligned}
\right.
\end{equation} 
\vspace{0.2cm}

\noindent where $f_1\in L^2(0,T;L^2(\Omega_f))$ and
$f_2\in L^2(0,T;L^2(\Omega_s))$ are the source functions. Here, the state space $\mathbf{H}$ is defined for the solution $(w(t), w_t(t), u(t))$ as
\begin{equation}\label{SEW}
\mathbf{H} 
\equiv
H^1(\Omega_s)\times L^2(\Omega_s)\times L^2(\Omega_f)
\end{equation}
with the inner product 
\begin{equation}\label{RTW}
\Big(
[w_0, w_{01},u],
[\tilde w_0,\tilde w_{01},\tilde u]
\Big)_{\mathbf{H}}
=
(w_0,\tilde w_0)_{H^1(\Omega_s)}
+
(w_{01},\tilde w_{01})_{L^2(\Omega_s)}
+
(u,\tilde u)_{L^2(\Omega_f)}
\end{equation}
where
\begin{equation}\label{WQA}
(w_0,\tilde w_0)_{H^1(\Omega_s)}
=
\int_{\Omega_s}
\big[\nabla w_0\cdot \nabla \tilde w_0 + w_0 \tilde w_0\big]\,
d\Omega_s.
\end{equation}

\section{Preliminaries}
\label{section2}
\vspace{0.3cm}

\noindent The PDE system \eqref{1.0} is a simplified version of the interactive dynamics of coupled 3D Stokes fluid and 3D Lame elastic structure PDE system in \cite{AR} where heat equation replaces the original Stokes equation. In \cite{AR}, a non-trivial semigroup formulation was used to show the existence-uniqueness of the solution. Since our PDE model \eqref{1.0} is the canonical case of the coupled system in \cite{AR}, a similar semigroup formulation can be used for PDE model \eqref{1.0} to show the existence-uniqueness of the solution. In our theretical error analysis, for the sake of simplicity, we will take the source functions $f_1=0$ and $f_2=0.$ Also, for clarity, we give details below concerning the semigroup formulation of our coupled PDE system \eqref{1.0}: it can be associated as an abstract ODE system in Hilbert space $\mathbf H$ (in \eqref{SEW}) as

\begin{equation}
\label{ode}
\left\{ 
\begin{array}{l}
\frac{d}{dt}\Phi (t)=\mathbf{A}\Phi (t)\\
\Phi (0)=\Phi _{0}
\end{array}%
\right. 
\end{equation}
where $\Phi (t)=\left(w(t),w_{t}(t),u(t)\right)
$ and $\Phi _{0}=(w_{0},w_{01},u_{0}).$ Here, the (generator) operator   $\mathbf{A}:D(\mathbf{A})\subset 
\mathbf{H}\rightarrow \mathbf{H}$ is defined as
\begin{equation}\label{eq:1.14a}
\mathbf{A}
\equiv
\begin{bmatrix}
0 & I & 0 \\[6pt]
\Delta - I & 0 & 0 \\[6pt]
0 & 0 & \Delta
\end{bmatrix},
\end{equation}

\noindent
\noindent
and the domain \(D(\mathbf{A})\subset \mathbf{H}\) is characterized as follows:\\ \\
\begin{equation}
\begin{aligned}
D(\mathbf{A})=\Big\{(z_1&,z_2,u)\in \mathbf{H}:\qquad  \nonumber\\
\text{\rm (i)}\quad
&z_1\in (H^1(\Omega_s))^d
\text{ with }\Delta z_1\in (L^2(\Omega_s))^d,
\text{ so that} \nonumber\\
&\frac{\partial z_1}{\partial \nu}\Big|_{\Gamma_s}
\in (H^{-1/2}(\Gamma_s))^d; \nonumber\\
\text{\rm (ii)}\quad
&z_2\in (H^1(\Omega_s))^d; \nonumber\\
\text{\rm (iii)}\quad
&u\in (H^1(\Omega_f))^d
\text{ with }\Delta u\in (L^2(\Omega_f))^d; \nonumber\\
\text{\rm (iv)}\quad
&\frac{\partial u}{\partial \nu}\Big|_{\Gamma_s}
\in (H^{-1/2}(\Gamma_s))^d; \nonumber\\
\text{\rm (v)}\quad
&\frac{\partial u}{\partial \nu}\Big|_{\Gamma_s}
=
\frac{\partial z_1}{\partial \nu}\Big|_{\Gamma_s}
\in (H^{-1/2}(\Gamma_s))^d; \nonumber\\
\text{\rm (vi)}\quad
&u|_{\Gamma_f}=0,\qquad
z_2|_{\Gamma_s}=u|_{\Gamma_s}
\in (H^{1/2}(\Gamma_s))^d\Big\}.
\end{aligned}
\end{equation}

\noindent One can easily use the linear  semigroup approach (as followed in \cite{AT, AR}) to show the existence uniqueness of the solution to the coupled PDE system \eqref{1.0}. This result is given as follows:
\begin{theorem}\label{thm-1.2}
With reference to the PDE model \eqref{1.0}, or the ODE system \eqref{ode}, the map
$$
\Phi_0=\{w_0,w_{01},u_0\}\longmapsto \Phi(t)=\{w(t),w_t(t),u(t)\}
$$
defines a strongly continuous contraction semigroup $\{e^{\mathbf{A}t}\}_{t\geq 0}$ on the energy space
$\mathbf H.$ Hence, for any initial data $\Phi_0=\{w_0,w_{01},u_0\}\in \mathbf H$,
there exists a unique solution 
\begin{equation}\label{eq}
\Phi(t)=\{w(t),w_t(t),u(t)\}\in C\big([0,T];\mathbf H\big)
\end{equation}
to the problem \eqref{1.0}, or \eqref{ode}. Moreover, if the initial data $\Phi_0=\{w_0, w_{01},u_0\}\in D(\mathbf{A})$ then \begin{align}
\label{8C}
\Phi(t)=\{w(t),w_t(t),u(t)\} \in C\big([0,T];D(\mathbf{A})\big).
\end{align}
\end{theorem}
\noindent In the context of this paper,  in order to obtain the necessary error bounds to determine how accurate our PINN-approximation algorithm, we will critically appeal to the \textit{``higher regularity"} result proved in \cite[Theorem 2.1]{AG} which is given for the Stokes-wave interaction system (heat equation in \eqref{1.0} is replaced with Stokes equation). For the sake of clarity and the readers' convenience we remind this result below: 
\begin{theorem}\label{thm2.1}
Let the initial data $\{w_0, w_{01},u_0\} \in D(\tilde{A})$, where $\tilde{A}$ is the generator of a Stokes-wave coupled PDE system considered in \cite{AG}. Also, assume that $
w_0\in H^2(\Omega_s).$
Then the solution $(w(t),w_t(t),u(t))$ given in \cite{AG} satisfies the following higher regularity with continuous dependence on the initial data:
\begin{align}
w &\in L^\infty(0,T;H^2(\Omega_s)), \label{eqw}\\
u &\in L^2(0,T;H^2(\Omega_f)), \label{equ}\\
p & \in L^2(0,T;H^1(\Omega_f)). 
\end{align}
\end{theorem} 

\section{Neural Networks for the PDE system \eqref{1.0}}
\label{4.1N}
\noindent Our main goal in this manuscript is to develop a PINN approximation framework for the coupled heat-wave system \eqref{1.0}. To this end, this section introduces a neural network architecture employed in our analysis and defines the terminology required for the error estimates that follow throughout the manuscript. We begin with the following definition for the neural networks.

\begin{definition}
\label{NN}
Let $R\in(0,\infty)$, $L,W\in\mathbb{N}$ with $L\ge2$.
Fix input and output dimensions $l_0,l_L\in\mathbb{N}$, and let the hidden widths
$l_j\in\{1,\ldots,W\},\; 1\le j\le L-1$, and
 $\sigma:\mathbb{R}\to\mathbb{R}$ be a $C^4$ activation function. Define the parameter set
$$
\Theta_{L,W,R}(l_0,l_L)
:=
\bigcup_{\,l_1,\ldots,l_{L-1}\in\{1,\ldots,W\}}
\ \prod_{j=1}^{L}
\Big( [-R,R]^{\,l_j\times l_{j-1}} \times [-R,R]^{\,l_j}\Big).
$$
For $\theta\in\Theta_{L,W,R}(l_0,l_L)$, denote $\theta_j=(W_j,b_j)$ and define the affine maps
$$
A_j^\theta:\mathbb{R}^{l_{j-1}}\to\mathbb{R}^{l_j},
\qquad
A_j^\theta(z)=W_j z + b_j,
\qquad 1\le j\le L,
$$
together with the layer maps
$$
f_j^\theta:\mathbb{R}^{l_{j-1}}\to\mathbb{R}^{l_j},
\qquad
f_j^\theta(z)=
\begin{cases}
(\sigma\circ A_j^\theta)(z), & 1\le j<L,\\[1mm]
A_L^\theta(z), & j=L.
\end{cases}
$$
We define the composed network $\varPhi_\theta:\mathbb{R}^{l_0}\to\mathbb{R}^{l_L}$ by
$
\varPhi_\theta(z) =( f_L^\theta\circ f_{L-1}^\theta\circ\cdots\circ f_1^\theta)(z),
$
for all $z \in \mathbb{R}^{l_0}$.
\end{definition}
\begin{remark}
For $1\le j\le L$, layer $j$ has width $l_j$ and $(W_j,b_j)$ are its weights and biases.
The first $L-1$ layers are the hidden layers, and the network width is $\max\{l_0,\ldots,l_L\}$.
If $L=2$ the network is shallow; if $L\ge3$ it is deep.
The $C^4$ requirement on $\sigma$ ensures that the PDE residuals, which involve fourth-order
derivatives of $u_\theta$ and $w_\theta$, are classically defined. 
\end{remark}
\noindent With the geometry introduced in Section~\ref{section2}, we now specialize
Definition~\ref{NN} to the PDE system \eqref{1.0}. Set
$$
l_0=d+1,\qquad z=(t,x)\in\mathbb{R}^{d+1},\qquad l_L=2d,
$$
so that $\varPhi_\theta:\mathbb{R}^{d+1}\to\mathbb{R}^{2d}$.
Define $\pi_u,\pi_w:\mathbb{R}^{2d}\to\mathbb{R}^{d}$ by
\begin{align*}
\pi_u(\xi_1,\ldots,\xi_{2d}) := (\xi_1,\ldots,\xi_d),
\qquad
\pi_w(\xi_1,\ldots,\xi_{2d}) := (\xi_{d+1},\ldots,\xi_{2d}),
\end{align*}
and define the domain-restricted network components.
\begin{align*}
u_\theta := (\pi_u\circ \varPhi_\theta)\big\lvert_{(0,T)\times\Omega_f} :
(0,T)\times\Omega_f \to \mathbb{R}^{d},
\quad
w_\theta := (\pi_w\circ \varPhi_\theta)\big\rvert_{(0,T)\times\Omega_s} :
(0,T)\times\Omega_s \to \mathbb{R}^{d}.
\end{align*}
Note that both components $u_\theta$ and $w_\theta$ are derived from a single shared network
$\varPhi_\theta$, evaluated on their respective subdomains. The heat residual uses $u_\theta$ on
$(0,T)\times\Omega_f$, the wave residual uses $w_\theta$ on $(0,T)\times\Omega_s$, and all
interface residuals are evaluated by restricting $\varPhi_\theta$ to $(0,T)\times\Gamma_s$. The time derivative $w_{\theta\, t}$ is computed using automatic differentiation.\\
\subsection{Quadrature rules}
\label{quad}
\medskip

\noindent In the PINN approximation framework, the generalization error is defined through integrals
of residuals over space-time and boundary domains. Since these integrals
cannot be evaluated exactly in practice, we approximate them numerically
using midpoint quadrature rules. This allows us to replace the continuous
loss functional by a computable discrete sum that is used in training.
\medskip

\noindent
We consider integrals of a function $f$ over the physical sets
$
(0,T)\times\Omega_f,\;
(0,T)\times\Omega_s,\;
(0,T)\times\Gamma_f,\;
(0,T)\times\Gamma_s,\;
\Omega_f\;\textrm{and}\;
\Omega_s.
$
Let $\Lambda$ denote any one of the above sets, whose dimension is $d+1$ for
$\Lambda=(0,T)\times\Omega_f$ or $(0,T)\times\Omega_s$, and $d$ for the remaining sets.\\

\noindent
To construct the quadrature, let $\widetilde{\Lambda}$ be a closed axis-aligned box containing
$\Lambda$, and introduce a uniform Cartesian grid on $\widetilde{\Lambda}$ with
$\{y_i\}_{i=1}^M \subset \widetilde{\Lambda}$ the midpoints of the grid cells.

\medskip

\noindent
Since the Cartesian grid on $\widetilde{\Lambda}$ does not in general conform to $\Lambda$,
we restrict the quadrature to $\Lambda$ by means of an indicator function
$$
\mathbf{1}_\Lambda(y_i)
=
\begin{cases}
1, & y_i \in \Lambda,\\
0, & y_i \notin \Lambda.
\end{cases}
$$
\noindent
The midpoint quadrature rule used in this work is therefore
$
Q_M^\Lambda[f]
:=
\sum_{i=1}^M \mathbf{1}_\Lambda(y_i)\, f(y_i)\, \lvert K\rvert,
$
where only the midpoints lying inside $\Lambda$ contribute to the sum. Depending on the type of domain $\Lambda$, the cell measure $\lvert K\rvert$ is given by
$$
\lvert K\rvert=
\begin{cases}
\displaystyle\prod_{i=1}^{d}\Delta x_i\,\,\Delta t, & \Lambda \subset (0,T)\times\Omega,\\[6pt]
\displaystyle\prod_{i=1}^{d}\Delta x_i, & \Lambda \subset \Omega,\\[6pt]
\Delta s\,\Delta t, & \Lambda \subset (0,T)\times\Gamma,\\[4pt]
\Delta s, & \Lambda \subset \Gamma,
\end{cases}
$$
where $\Delta x_1,\ldots,\Delta x_d,\Delta t$ denote the Cartesian grid spacings and
$\Delta s$ is the uniform spacing of the quadrature nodes along $\Gamma$. For $N\in\mathbb N$, divide each coordinate direction of
$\widetilde{\Lambda}$ into $N$ equal subintervals. If
$L_1,\ldots,L_{ d}$ denote the side lengths of
$\widetilde{\Lambda}$, then the corresponding grid spacings are
$
\Delta y_j=\frac{L_j}{N},\; j=1,\ldots,d.
$\\

\noindent Define also the quadrature spacing  by
$
h:=\max_{1\le j\le d}\Delta y_j
=
\frac{\max_{1\le j\le  d}L_j}{N}
$ with $
M=N^{ d}
$. For the $(d+1)$-dimensional space-time interior domains, $d$ in these expressions is replaced by $d+1$.
The integral over $\Lambda$ is approximated by
$$
\int_\Lambda f
\;\approx\;
Q_M^\Lambda[f]
=
\sum_{i=1}^M \mathbf{1}_\Lambda(y_i)\, f(y_i)\, \lvert K\rvert.
$$
Since the curved physical domains are not aligned with the Cartesian
background grid, the standard midpoint accuracy degenerates to first
order near $\partial\Lambda$; see,
e.g., \cite{LZ}. 
\noindent
If $g\in C^2(\Lambda)$, then the quadrature error satisfies
\begin{align}
\label{quad_err}
\left\lvert
\int_\Lambda g-Q_M^\Lambda[g]
\right\rvert
\le
C_\Lambda
\lVert g\rVert_{C^2(\Lambda)}
M^{-1/d}\lesssim \lVert g\rVert_{C^2(\Lambda)}.
\end{align}
For the $(d+1)$-dimensional space-time interior domains, $M^{-1/d}$ in
\eqref{quad_err} is replaced by $M^{-1/(d+1)}$.\\

\subsection{Physics-Informed Neural Networks}
\noindent Physics-informed neural networks (PINNs) provide a framework for approximating solutions of partial differential equations using feedforward neural networks. For the coupled system \eqref{1.0}, the residuals consist of the interior residuals in the domains $\Omega_f$ and $\Omega_s$, the outer boundary residual on $\Gamma_f$, the interface residuals on $\Gamma_s$, and the initial residuals at $t=0$.
\medskip

\noindent We now describe the PINN framework for the PDE system \eqref{1.0}.  Let $(w_\theta,w_{\theta t}, u_\theta)$ be a neural network pair parameterized by $\theta\in\Theta$ approximating the solution $(w,w_t, u)$ of \eqref{1.0}. The residuals associated with the coupled heat-wave system are defined by
\begin{align}
\text{Interior:}\quad
&\mathcal R_u := \partial_t u_\theta - \Delta u_\theta 
\quad \text{in } (0,T)\times\Omega_f,
\label{resi}\\
&\mathcal R_w := \partial_{tt} w_\theta - \Delta w_\theta + w_\theta 
\quad \text{in } (0,T)\times\Omega_s,\\[1ex]
\text{Boundary:}\quad
&\mathcal R_{BC} := u_\theta 
\quad \text{on } (0,T)\times\Gamma_f,\\[1ex]
\text{Interface:}\quad
&\mathcal R_{IF1} := u_\theta - \partial_t w_\theta
\quad \text{on } (0,T)\times\Gamma_s,\\
&\mathcal R_{IF2} := \partial_\nu u_\theta - \partial_\nu w_\theta
\quad \text{on } (0,T)\times\Gamma_s,\\[1ex]
\text{Initial:}\quad
&\mathcal R_{IC1} := w_\theta(0,\cdot) - w_0 
\quad \text{in } \Omega_s,\\
&\mathcal R_{IC2} := \partial_t w_\theta(0,\cdot) - w_{01} 
\quad \text{in } \Omega_s,\\
&\mathcal R_{IC3} := u_\theta(0,\cdot) - u_0 
\quad \text{in } \Omega_f.
\label{resIC}
\end{align}
These residuals quantify the extent to which $(w_\theta, w_{\theta t}, u_\theta,)$ satisfies the coupled system \eqref{1.0}. In particular, for the exact solution $(w,w_t, u)$, all residuals vanish.
\medskip

\noindent
Having defined the residuals above, we next derive an energy estimate for the $L^2$-error between the exact solution $(w,w_t, u)$ and the PINN approximation $(w_\theta, w_{\theta t}, u_\theta)$ which provides an affirmative answer to Question~\textbf{[Q1]} posed in Section 1.

\begin{theorem}
\label{4.2S}
Assume the hypotheses of Theorem~\ref{thm2.1} hold for $d\in \{2,3\}$. Let $(w,w_t, u)$ be the exact solution
of the PDE system \eqref{1.0}, and let $(w_{\theta},w_{\theta t}, u_{\theta})$ be the corresponding PINN approximation. Then, for $
\hat{u}:=u_\theta-u$
$
\hat{w}:=w_\theta-w$, and $\hat{w}_t=w_{\theta t}-w_t$  
\begin{align}
\lVert \hat u\rVert_{L^2(\Omega_f\times[0,T])}^2
+ \lVert \partial_t \hat w\rVert_{L^2(\Omega_s\times[0,T])}^2
+ \lVert \hat w\rVert_{H^1(\Omega_s\times[0,T])}^2
\le C_{\mathcal{R}}\,(T + T^2 e^{T})
\label{22E}
\end{align}
where \begin{align}
\label{aslm}
C_{\mathcal{R}} &:=
\lVert\mathcal{R}_{IC1}\rVert_{H^1(\Omega_s)}^2
+\lVert\mathcal{R}_{IC2}\rVert_{L^2(\Omega_s)}^2
+\lVert\mathcal{R}_{IC3}\rVert_{L^2(\Omega_f)}^2
+\lVert\mathcal{R}_{u}\rVert_{L^2(\Omega_f\times[0,T])}^{2}
+\lVert\mathcal{R}_{w}\rVert_{L^2(\Omega_s\times[0,T])}^{2}
\nonumber\\
&\quad
+C\lVert\mathcal{R}_{BC}\rVert_{L^2(\Gamma_f\times[0,T])}
+C\lVert\mathcal{R}_{IF1}\rVert_{L^2(\Gamma_s\times[0,T])}
+C\lVert\mathcal{R}_{IF2}\rVert_{L^2(\Gamma_s\times[0,T])}
\nonumber\\
&\quad
+\lVert\mathcal{R}_{IF1}\rVert_{L^2(\Gamma_s\times[0,T])}^{2}
+\lVert\mathcal{R}_{IF2}\rVert_{L^2(\Gamma_s\times[0,T])}^{2}.
\end{align}
\end{theorem}
\begin{proof}
We begin by comparing the PINN approximation with the exact solution of
\eqref{1.0}. Subtracting the exact solutions from the PINN solutions, the error
$(\hat{w},\hat{w}_t ,\hat{u})$ satisfies
\begin{equation}
\label{1000}
\left\{
\begin{aligned}
\hat{u}_t - \Delta \hat{u} &= \mathcal{R}_u&&\text{in } (0,T)\times\Omega_f, \\[4pt]
\hat{w}_{tt} - \Delta \hat{w} + \hat{w} &= \mathcal{R}_w &&\text{in } (0,T)\times\Omega_s, \\[6pt]
\hat{u}|_{\Gamma_f} &= \mathcal{R}_{BC} &&\text{on } (0,T)\times\Gamma_f, \\[4pt]
\hat{u}-\hat{w}_t &=\mathcal{R}_{IF1}&&\text{on } (0,T)\times\Gamma_s, \\[3pt]
\dfrac{\partial \hat{u}}{\partial \nu} - \dfrac{\partial \hat{w}}{\partial \nu}&=\mathcal{R}_{IF2}
&&\text{on } (0,T)\times\Gamma_s, \\[6pt]
[\hat{w}(0,\cdot),\, \hat{w}_t(0,\cdot),\, \hat{u}(0,\cdot)]
&= [\mathcal{R}_{IC1},\mathcal{R}_{IC2}, \mathcal{R}_{IC3}]
&&\text{in } \mathbf{H}.
\end{aligned}
\right.
\end{equation}
Multiplying $\eqref{1000}_1$ and $\eqref{1000}_2$ with $\hat{u}$ and $\hat{w}_t,$ respectively and using Green's theorem gives
\begin{align}
\label{CMO}
\frac{1}{2}\frac{d}{dt} \Big(\lVert \hat{u}\rVert_{L^2(\Omega_f)}^{2}
&+\lVert \hat{w}_t\rVert_{L^2(\Omega_s)}^{2}
+\lVert \nabla \hat{w}\rVert_{L^2(\Omega_s)}^{2}
+\lVert \hat{w}\rVert_{L^2(\Omega_s)}^{2}\Big)
+\lVert \nabla \hat{u}\rVert_{L^2(\Omega_f)}^{2}
\nonumber\\
&= \int_{\Omega_f} \mathcal{R}_u \cdot \hat{u} \,dx
+\int_{\Omega_s}\mathcal{R}_w \cdot \hat{w}_t \,dx
+\int_{\Gamma_f} \hat{u} \frac{\partial \hat{u}}{\partial \nu} \,d \Gamma_f \nonumber\\
&+\int_{\Gamma_s}\hat{w}_t \frac{\partial \hat{w}}{\partial \nu}\,d\Gamma_s 
-\int_{\Gamma_s} \hat{u} \, \frac{\partial \hat{u}}{\partial \nu}\,d\Gamma_s.
\end{align}
For the last term on the RHS of \eqref{CMO}: we use the interface relations $$\hat{u} = \hat{w}_t + \mathcal{R}_{IF1}; \qquad
\partial_\nu\hat{u} = \partial_\nu\hat{w} + \mathcal{R}_{IF2}$$ to get
\begin{align}
\int_{\Gamma_s}\hat{u}\,\frac{\partial\hat{u}}{\partial\nu}\, d \Gamma_s
&= \int_{\Gamma_s}(\hat{w}_t+\mathcal{R}_{IF1})(\partial_\nu\hat{w}+\mathcal{R}_{IF2})\, d \Gamma_s \nonumber
\\
&= \int_{\Gamma_s}\hat{w}_t\,\partial_\nu\hat{w}\,d \Gamma_s
+\int_{\Gamma_s}\hat{w}_t\,\mathcal{R}_{IF2}\,d \Gamma_s \nonumber\\
&+\int_{\Gamma_s}\mathcal{R}_{IF1}\,\partial_\nu\hat{w}\,d \Gamma_s
+\int_{\Gamma_s}\mathcal{R}_{IF1}\,\mathcal{R}_{IF2}\,d \Gamma_s.
\label{asKL}
\end{align}
Substituting \eqref{asKL} into \eqref{CMO}, and  seeing the cancellation of the terms
$\int_{\Gamma_s}\hat{w}_t\,\partial_\nu\hat{w}$  we obtain
\begin{align}
\label{CLND}
\frac{1}{2}\frac{d}{dt} \Big(\lVert \hat{u}\rVert_{L^2(\Omega_f)}^{2}
&+\lVert \hat{w}_t\rVert_{L^2(\Omega_s)}^{2}
+\lVert \hat{w}\rVert_{H^1(\Omega_s)}^{2}\Big)
+\lVert \nabla \hat{u}\rVert_{L^2(\Omega_f)}^{2}
\nonumber\\
&= \int_{\Omega_f} \mathcal{R}_u \cdot \hat{u} \,dx
+\int_{\Omega_s}\mathcal{R}_w \cdot \hat{w}_t \,dx
+\int_{\Gamma_f} \hat{u} \frac{\partial \hat{u}}{\partial \nu} \,d \Gamma_f
\nonumber\\
&\quad
-\int_{\Gamma_s}\hat{w}_t\,\mathcal{R}_{IF2}\,d\Gamma_s
-\int_{\Gamma_s}\mathcal{R}_{IF1}\,\partial_\nu\hat{w}\,d\Gamma_s
-\int_{\Gamma_s}\mathcal{R}_{IF1}\,\mathcal{R}_{IF2}\,d\Gamma_s,
\end{align}
We now apply H\"{o}lder's and Young's inequalities to estimate the first and second terms on the right-hand side of \eqref{CLND} to obtain
\begin{align}
\left\lvert \int_{\Omega_f}\mathcal{R}_u\cdot\hat{u}\right\rvert
&\le \tfrac{1}{2}\lVert\mathcal{R}_u\rVert_{L^2(\Omega_f)}^2
   +\tfrac{1}{2}\lVert\hat{u}\rVert_{L^2(\Omega_f)}^2,
   \label{edr}
\\
\left\lvert \int_{\Omega_s}\mathcal{R}_w\cdot\hat{w}_t\right\rvert
&\le \tfrac{1}{2}\lVert\mathcal{R}_w\rVert_{L^2(\Omega_s)}^2
   +\tfrac{1}{2}\lVert\hat{w}_t\rVert_{L^2(\Omega_s)}^2.
   \label{s2dw}
\end{align}
For the third term on the right-hand side of \eqref{CLND}, using
$\hat{u}|_{\Gamma_f}=\mathcal{R}_{BC}$ from \eqref{1000}$_3$, applying
H\"{o}lder's inequality and the Sobolev trace theorem we get
\begin{align}
\left\lvert \int_{\Gamma_f}\hat{u}\,\frac{\partial\hat{u}}{\partial\nu}\right\rvert
&\le \lVert\mathcal{R}_{BC}\rVert_{L^2(\Gamma_f)}
   \lVert\partial_\nu\hat{u}\rVert_{L^2(\Gamma_f)}
\nonumber\\
&\le C\lVert\mathcal{R}_{BC}\rVert_{L^2(\Gamma_f)}
   \lVert\hat{u}\rVert_{H^2(\Omega_f)}.
   \label{po3}
\end{align}
For the fourth, fifth, and sixth terms on the right-hand side of \eqref{CLND}, again applying H\"{o}lder's inequality and the Sobolev trace theorem yield
\begin{align}
\left\lvert \int_{\Gamma_s}\hat{w}_t\,\mathcal{R}_{IF2}\right\rvert
&\le \lVert\mathcal{R}_{IF2}\rVert_{L^2(\Gamma_s)}
   \lVert\hat{w}_t\rVert_{L^2(\Gamma_s)}
   \nonumber\\
&\le C\lVert\mathcal{R}_{IF2}\rVert_{L^2(\Gamma_s)}
   \lVert\hat{w}_t\rVert_{H^1(\Omega_s)},\label{yhx}\\
\left\lvert \int_{\Gamma_s}\mathcal{R}_{IF1}\,\partial_\nu\hat{w}\right\rvert
&\le \lVert\mathcal{R}_{IF1}\rVert_{L^2(\Gamma_s)}
   \lVert\partial_\nu\hat{w}\rVert_{L^2(\Gamma_s)}
\nonumber\\
&\le C\lVert\mathcal{R}_{IF1}\rVert_{L^2(\Gamma_s)}
   \lVert\hat{w}\rVert_{H^2(\Omega_s)},\label{fgt}
\\
\left\lvert \int_{\Gamma_s}\mathcal{R}_{IF1}\,\mathcal{R}_{IF2}\right\rvert
&\le \tfrac{1}{2}\lVert\mathcal{R}_{IF1}\rVert_{L^2(\Gamma_s)}^2
   +\tfrac{1}{2}\lVert\mathcal{R}_{IF2}\rVert_{L^2(\Gamma_s)}^2.
   \label{kl4}
\end{align}
Using the preceding estimates \eqref{edr}-\eqref{kl4} in \eqref{CLND},
and then integrating over $[0,t]$,
applying H\"{o}lder's inequality in time and dropping the nonnegative term
$\int_0^t\lVert\nabla\hat{u}(s)\rVert_{L^2(\Omega_f)}^2\,ds$, we obtain
\begin{align}
\label{lcx2}
\lVert\hat{u}(t)\rVert_{L^2(\Omega_f)}^2
+\lVert\hat{w}_t(t)\rVert_{L^2(\Omega_s)}^2
+\lVert\hat{w}(t)\rVert_{H^1(\Omega_s)}^2
&\le \lVert\hat{u}(0)\rVert_{L^2(\Omega_f)}^2
+\lVert\hat{w}_t(0)\rVert_{L^2(\Omega_s)}^2
+\lVert\hat{w}(0)\rVert_{H^1(\Omega_s)}^2
\nonumber\\
&+\int_0^t \Big(
\lVert\hat{u}(s)\rVert_{L^2(\Omega_f)}^2
+\lVert\hat{w}_t(s)\rVert_{L^2(\Omega_s)}^2
+\lVert\hat{w}(s)\rVert_{H^1(\Omega_s)}^2
\Big)\,ds
\nonumber\\
&+\lVert\mathcal{R}_u\rVert_{L^2(\Omega_f\times[0,T])}^2
+\lVert\mathcal{R}_w\rVert_{L^2(\Omega_s\times[0,T])}^2
\nonumber\\
&
+2C\lVert\mathcal{R}_{BC}\rVert_{L^2(\Gamma_f\times[0,T])}
\lVert\hat{u}\rVert_{L^2(0,T;H^2(\Omega_f))}
\nonumber\\
&
+2C\lVert\mathcal{R}_{IF1}\rVert_{L^2(\Gamma_s\times[0,T])}
\lVert\hat{w}\rVert_{L^2(0,T;H^2(\Omega_s))}
\nonumber\\
&
+2C\lVert\mathcal{R}_{IF2}\rVert_{L^2(\Gamma_s\times[0,T])}
\lVert\hat{w}_t\rVert_{L^2(0,T;H^1(\Omega_s))}
\nonumber\\
&
+\lVert\mathcal{R}_{IF1}\rVert_{L^2(\Gamma_s\times[0,T])}^2
+\lVert\mathcal{R}_{IF2}\rVert_{L^2(\Gamma_s\times[0,T])}^2,
\end{align}
Using now the initial condition
\begin{align}
\lVert\hat{w}(0)\rVert_{H^1(\Omega_s)}^2
+\lVert\hat{w}_t(0)\rVert_{L^2(\Omega_s)}^2
+\lVert\hat{u}(0)\rVert_{L^2(\Omega_f)}^2=
\lVert\mathcal{R}_{IC1}\rVert_{H^1(\Omega_s)}^2
+\lVert\mathcal{R}_{IC2}\rVert_{L^2(\Omega_s)}^2
+\lVert\mathcal{R}_{IC3}\rVert_{L^2(\Omega_f)}^2
\end{align}
together with the higher regularity results
\eqref{8C}, \eqref{eqw} and \eqref{equ} given in Theorem \ref{thm2.1} in \eqref{lcx2}, we take
\begin{align}
\lVert\hat{u}(t)\rVert_{L^2(\Omega_f)}^2
+\lVert\hat{w}_t(t)\rVert_{L^2(\Omega_s)}^2
+\lVert\hat{w}(t)\rVert_{H^1(\Omega_s)}^2
&\le \int_0^t \Big(
\lVert\hat{u}(s)\rVert_{L^2(\Omega_f)}^2
+\lVert\hat{w}_t(s)\rVert_{L^2(\Omega_s)}^2
+\lVert\hat{w}(s)\rVert_{H^1(\Omega_s)}^2
\Big)\,ds \nonumber\\
&+ C_{\mathcal{R}},
\label{36H}
\end{align}
where
\begin{align}
\label{C1def}
C_{\mathcal{R}} &:=
\lVert\mathcal{R}_{IC1}\rVert_{H^1(\Omega_s)}^2
+\lVert\mathcal{R}_{IC2}\rVert_{L^2(\Omega_s)}^2
+\lVert\mathcal{R}_{IC3}\rVert_{L^2(\Omega_f)}^2
+\lVert\mathcal{R}_{u}\rVert_{L^2(\Omega_f\times[0,T])}^{2}
+\lVert\mathcal{R}_{w}\rVert_{L^2(\Omega_s\times[0,T])}^{2}
\nonumber\\
&\quad
+C\lVert\mathcal{R}_{BC}\rVert_{L^2(\Gamma_f\times[0,T])}
+C\lVert\mathcal{R}_{IF1}\rVert_{L^2(\Gamma_s\times[0,T])}
+C\lVert\mathcal{R}_{IF2}\rVert_{L^2(\Gamma_s\times[0,T])}
\nonumber\\
&\quad
+\lVert\mathcal{R}_{IF1}\rVert_{L^2(\Gamma_s\times[0,T])}^{2}
+\lVert\mathcal{R}_{IF2}\rVert_{L^2(\Gamma_s\times[0,T])}^{2}.
\end{align}
\noindent Application of Gr\"{o}nwall's inequality to \eqref{36H} implies 
\begin{align}
\lVert \hat{u}(t)\rVert_{L^2(\Omega_f)}^{2}
+\lVert \hat{w}_t(t)\rVert_{L^2(\Omega_s)}^{2}
+\lVert \hat{w}(t)\rVert_{H^1(\Omega_s)}^{2}
\le C_{\mathcal{R}}(1+Te^{T}),
\label{31ES}
\end{align}
and integrating \eqref{31ES} over $[0,T]$ finally gives
\begin{align}
 \lVert \hat{u}\rVert_{L^2(\Omega_f\times[0,T])}^{2}
 +\lVert \hat{w}_t\rVert_{L^2(\Omega_s\times[0,T])}^{2}
 +\lVert \hat{w}\rVert_{H^1(\Omega_s\times[0,T])}^{2}
 \le C_{\mathcal{R}}\,(T+T^2e^{T}),
\end{align}
and this completes the proof of Theorem \ref{4.2S}.
\end{proof}
\noindent Having established an energy estimate for the $L^2$-error between the exact 
solution $(w,w_t,u)$ and the PINN approximation 
$(w_\theta,w_{\theta t},u_\theta)$, we now turn our attention to Question~\textbf{[Q2]}. In order to address this question, we will need to obtain 
bounds on the derivatives of the neural network. The following lemma, which provides those bounds, will be one of the key ingredients of the proof of our main result given in Theorem~\ref{Th 3.3A} below.
\begin{lemma}
\label{lemC1}
Let $d\in\{2,3\}$, $n,L,W\in\mathbb N$, $L\ge2$, and
$R\in[1,\infty)$. Let
$\pi_u:\mathbb R^{2d}\to\mathbb R^d$ and
$\pi_w:\mathbb R^{2d}\to\mathbb R^d$
denote the projections onto the first and last $d$ components,
respectively:
\begin{align*}
\pi_u(z_1,\dots,z_{2d})
&=
(z_1,\dots,z_d),
\\
\pi_w(z_1,\dots,z_{2d})
&=
(z_{d+1},\dots,z_{2d}).
\end{align*}
Assume $\sigma\in C^n(\mathbb R)$ and
$\lVert\sigma\rVert_{C^n(\mathbb R)}\ge1$.
Fix $l_0=d+1$ and $l_L=2d$, and let
$\theta\in\Theta_{L,W,R}(l_0,l_L)$
in Definition~\ref{NN}. Then for every component $1\le j\le2d$,
\begin{align}
\lVert(\varPhi_\theta)_j\rVert_{C^n(\mathbb R^{d+1})}
\le
16^L(d+1)^{2n}
\Big(
e^2n^4W^3R^n
\lVert\sigma\rVert_{C^n(\mathbb R)}
\Big)^{nL}.
\label{EqC1}
\end{align}
Consequently, defining
\begin{align*}
u_\theta
&:=
(\pi_u\circ\varPhi_\theta)
\big\lvert_{(0,T)\times\Omega_f},
\\
w_\theta
&:=
(\pi_w\circ\varPhi_\theta)
\big\rvert_{(0,T)\times\Omega_s},
\end{align*}
we have, for $1\le j\le d$,
\begin{align}
\lVert(u_\theta)_j\rVert_{C^n((0,T)\times\Omega_f)}
&\le
16^L(d+1)^{2n}
\Big(
e^2n^4W^3R^n
\lVert\sigma\rVert_{C^n(\mathbb R)}
\Big)^{nL},
\label{EqCU}\\
\lVert(w_\theta)_j\rVert_{C^n((0,T)\times\Omega_s)}
&\le
16^L(d+1)^{2n}
\Big(
e^2n^4W^3R^n
\lVert\sigma\rVert_{C^n(\mathbb R)}
\Big)^{nL}.
\label{EqCW}
\end{align}
\end{lemma}

\begin{proof}
For $1\le k\le L$, define the composition
$F_k:=f_L^\theta\circ f_{L-1}^\theta\circ\cdots\circ f_k^\theta,$ where the functions $f_k^\theta:
\mathbb R^{l_{k-1}}
\longrightarrow
\mathbb R^{l_k},
\,
1\le k\le L$ are the layer maps introduced in Definition \ref{NN}.
It is clear that $F_1=\varPhi_\theta$ and $F_L=f_L^\theta$. Hence, by definition of $F_k:
\mathbb R^{l_{k-1}}
\longrightarrow
\mathbb R^{l_L},$
and
$F_{k+1}:
\mathbb R^{l_k}
\longrightarrow
\mathbb R^{l_L},
$
\begin{align}
F_k
=
F_{k+1}\circ f_k^\theta.
\label{Fk2}
\end{align}
Therefore, applying Lemma~ \ref{x2} in Appendix
 to \eqref{Fk2}, for each
$1\le k\le L-2$, yields
\begin{equation}\label{eqC3}
\lVert F_k\rVert_{C^n(\mathbb R^{l_{k-1}})}
\le
16
\Big(
e^2n^4l_kl_{k-1}^2
\Big)^n
\Big(
\max_{1\le i\le l_k}
\lVert(f_k^\theta)_i
\rVert_{C^n(\mathbb R^{l_{k-1}})}
\Big)^n
\lVert F_{k+1}\rVert_{C^n(\mathbb R^{l_k})}.
\end{equation}
Also, since Definition~\ref{NN} gives
$f_k^\theta=\sigma\circ A_k^\theta$ $(1\le k<L)$ then for every
$1\le i\le l_k$,
\begin{equation}\label{eqlay}
\lVert(f_k^\theta)_i
\rVert_{C^n(\mathbb R^{l_{k-1}})}
\le
R^n\lVert\sigma\rVert_{C^n(\mathbb R)}.
\end{equation}
Now, it remains to estimate the last two layers. By definition
of $F_{L-1}$,
\begin{align}
F_{L-1}
=
f_L^\theta\circ f_{L-1}^\theta.
\label{FL1}
\end{align}
Since Definition~\ref{NN} gives
$f_L^\theta=A_L^\theta$ and
$A_L^\theta(z)=W_Lz+b_L$, for every
$1\le j\le l_L$ and $x\in\mathbb R^{l_{L-2}}$, we have
\begin{align}
(F_{L-1})_j(x)
&=
(f_L^\theta)_j
\big(f_{L-1}^\theta(x)\big)
\nonumber\\
&=
\big(
W_Lf_{L-1}^\theta(x)+b_L
\big)_j
\nonumber\\
&=
\sum_{i=1}^{l_{L-1}}
(W_L)_{ji}
(f_{L-1}^\theta)_i(x)
+
(b_L)_j.
\label{FL2}
\end{align}
Since $\theta\in\Theta_{L,W,R}(l_0,l_L)$, Definition~\ref{NN}
also implies
\begin{align}
\lvert(W_L)_{ji}\rvert
\le R,
\qquad
\lvert(b_L)_j\rvert
\le R,
\qquad
l_{L-1}\le W
\label{FL3}
\end{align}
and
$f_{L-1}^\theta=\sigma\circ A_{L-1}^\theta$ then for every
$x\in\mathbb R^{l_{L-2}}$ and $1\le i\le l_{L-1}$,
\begin{align}
\lvert(f_{L-1}^\theta)_i(x)\rvert
\le
\lVert\sigma\rVert_{C^0(\mathbb R)}.
\label{FL4}
\end{align}
Therefore, using \eqref{FL2}-\eqref{FL4}, we get
\begin{align}
\lvert(F_{L-1})_j(x)\rvert
&\le
\sum_{i=1}^{l_{L-1}}
\lvert(W_L)_{ji}\rvert
\lvert(f_{L-1}^\theta)_i(x)\rvert
+
\lvert(b_L)_j\rvert
\nonumber\\
&\le
Rl_{L-1}
\lVert\sigma\rVert_{C^0(\mathbb R)}
+
R\le
RW\lVert\sigma\rVert_{C^0(\mathbb R)}
+
R.
\label{FL5}
\end{align}
Taking the supremum over $x\in\mathbb R^{l_{L-2}}$ in
\eqref{FL5} gives
\begin{align}
\lVert(F_{L-1})_j
\rVert_{C^0(\mathbb R^{l_{L-2}})}
\le
RW\lVert\sigma\rVert_{C^0(\mathbb R)}
+
R.
\label{FL6}
\end{align}
Next, let $\alpha$ be a multi-index satisfying
$1\le\lvert\alpha\rvert\le n$. Differentiating \eqref{FL2} yields
\begin{align}
D^\alpha(F_{L-1})_j
=
\sum_{i=1}^{l_{L-1}}
(W_L)_{ji}
D^\alpha(f_{L-1}^\theta)_i.
\label{FL7}
\end{align}
Consequently, using \eqref{FL3} and then \eqref{eqlay} with
$k=L-1$, it follows from \eqref{FL7} that
\begin{align}
\lVert D^\alpha(F_{L-1})_j
\rVert_{C^0(\mathbb R^{l_{L-2}})}
&\le
Rl_{L-1}
\max_{1\le i\le l_{L-1}}
\lVert
D^\alpha(f_{L-1}^\theta)_i
\rVert_{C^0(\mathbb R^{l_{L-2}})}
\nonumber\\
&\le
RW
\max_{1\le i\le l_{L-1}}
\lVert
(f_{L-1}^\theta)_i
\rVert_{C^n(\mathbb R^{l_{L-2}})}
\le
WR^{n+1}
\lVert\sigma\rVert_{C^n(\mathbb R)}.
\label{FL8}
\end{align}
Since
$\lVert\sigma\rVert_{C^0(\mathbb R)}
\le
\lVert\sigma\rVert_{C^n(\mathbb R)}$
and
$W,R,\lVert\sigma\rVert_{C^n(\mathbb R)}\ge1$, it follows from
\eqref{FL6} that
\begin{align}
\lVert(F_{L-1})_j
\rVert_{C^0(\mathbb R^{l_{L-2}})}
&\le
RW\lVert\sigma\rVert_{C^n(\mathbb R)}
+
R
\nonumber\\
&\le
RW\lVert\sigma\rVert_{C^n(\mathbb R)}
+
WR^{n+1}\lVert\sigma\rVert_{C^n(\mathbb R)}
\le
2WR^{n+1}
\lVert\sigma\rVert_{C^n(\mathbb R)}.
\label{FL9}
\end{align}
Moreover, \eqref{FL8} gives, for every
$1\le\lvert\alpha\rvert\le n$,
\begin{align}
\lVert D^\alpha(F_{L-1})_j
\rVert_{C^0(\mathbb R^{l_{L-2}})}
\le
WR^{n+1}
\lVert\sigma\rVert_{C^n(\mathbb R)}
\le
2WR^{n+1}
\lVert\sigma\rVert_{C^n(\mathbb R)}.
\label{FL10}
\end{align}
Combining \eqref{FL9} and \eqref{FL10}, and then taking the maximum
over $1\le j\le l_L$, we obtain
\begin{equation}\label{eqlas}
\lVert F_{L-1}\rVert_{C^n(\mathbb R^{l_{L-2}})}
\le
2WR^{n+1}
\lVert\sigma\rVert_{C^n(\mathbb R)}.
\end{equation}
For $L\ge3$, repeatedly applying \eqref{eqC3} for
$k=1,2,\dots,L-2$ gives
\begin{equation}\label{it}
\lVert F_1\rVert_{C^n(\mathbb R^{d+1})}
\le
\lVert F_{L-1}\rVert_{C^n(\mathbb R^{l_{L-2}})}
\prod_{k=1}^{L-2}
\left[
16
\Big(
e^2n^4l_kl_{k-1}^2
\Big)^n
\Big(
\max_{1\le i\le l_k}
\lVert(f_k^\theta)_i
\rVert_{C^n(\mathbb R^{l_{k-1}})}
\Big)^n
\right].
\end{equation}
Since $F_1=\varPhi_\theta$, inserting \eqref{eqlay} and
\eqref{eqlas} into \eqref{it} yields
\begin{equation}\label{eqaf}
\lVert\varPhi_\theta\rVert_{C^n(\mathbb R^{d+1})}
\le
2WR^{n+1}
\lVert\sigma\rVert_{C^n(\mathbb R)}
\prod_{k=1}^{L-2}
16
\Big(
e^2n^4l_kl_{k-1}^2
\Big)^n
\Big(
R^n\lVert\sigma\rVert_{C^n(\mathbb R)}
\Big)^n.
\end{equation}
Here $l_0=d+1$ and
$l_k\le W$ for every $1\le k\le L-1$ by Definition~\ref{NN}. In particular,
\begin{align}
l_1l_0^2
\le
W(d+1)^2,
\qquad
l_kl_{k-1}^2
\le
W^3,
\quad
2\le k\le L-2.
\label{LW1}
\end{align}
Therefore, for $L\ge3$,
\begin{align}
\prod_{k=1}^{L-2}
(l_kl_{k-1}^2)^n
&\le
\big(W(d+1)^2\big)^n
(W^3)^{n(L-3)}
\nonumber\\
&=
(d+1)^{2n}
W^n
(W^3)^{n(L-3)}
\le
(d+1)^{2n}
(W^3)^{n(L-2)},
\label{eqwid}
\end{align}
where the last inequality follows from the fact that $W\ge1$. Inserting \eqref{eqwid} into \eqref{eqaf}  gives
\begin{equation}\label{U-bound}
\lVert\varPhi_\theta\rVert_{C^n(\mathbb R^{d+1})}
\le
2WR^{n+1}
\lVert\sigma\rVert_{C^n(\mathbb R)}
(d+1)^{2n}
\Big[
16
\Big(
e^2n^4W^3R^n
\lVert\sigma\rVert_{C^n(\mathbb R)}
\Big)^n
\Big]^{L-2}.
\end{equation}
Since again
$n,W,R,\lVert\sigma\rVert_{C^n(\mathbb R)}\ge1$, we have
\begin{align}
2WR^{n+1}
\lVert\sigma\rVert_{C^n(\mathbb R)}
\le
16^2
\Big(
e^2n^4W^3R^n
\lVert\sigma\rVert_{C^n(\mathbb R)}
\Big)^{2n}.
\label{RU1}
\end{align}
Substituting \eqref{RU1} into \eqref{U-bound}, we obtain
\begin{align}
\lVert\varPhi_\theta\rVert_{C^n(\mathbb R^{d+1})}
&\le
16^2 
\Big(
e^2n^4W^3R^n
\lVert\sigma\rVert_{C^n(\mathbb R)}
\Big)^{2n}
(d+1)^{2n}
\Big[
16
\Big(
e^2n^4W^3R^n
\lVert\sigma\rVert_{C^n(\mathbb R)}
\Big)^n
\Big]^{L-2}
\nonumber\\
&=
16^L(d+1)^{2n}
\Big(
e^2n^4W^3R^n
\lVert\sigma\rVert_{C^n(\mathbb R)}
\Big)^{nL}.
\label{EqU}
\end{align}
For $L=2$, we have $F_{L-1}=F_1=\varPhi_\theta$ and
$l_{L-2}=l_0=d+1$. Therefore, \eqref{eqlas} gives
\begin{align}
\lVert\varPhi_\theta\rVert_{C^n(\mathbb R^{d+1})}
\le
2WR^{n+1}
\lVert\sigma\rVert_{C^n(\mathbb R)}.
\label{L21}
\end{align}
Using \eqref{RU1} and $(d+1)^{2n}\ge1$ in \eqref{L21}, we obtain
\begin{align}
\lVert\varPhi_\theta\rVert_{C^n(\mathbb R^{d+1})}
&\le
16^2(d+1)^{2n}
\Big(
e^2n^4W^3R^n
\lVert\sigma\rVert_{C^n(\mathbb R)}
\Big)^{2n}.
\label{L22}
\end{align}
Thus, \eqref{EqU} holds for every $L\ge2$. Since
\[
\lVert\varPhi_\theta\rVert_{C^n(\mathbb R^{d+1})}
:=
\max_{1\le j\le2d}
\lVert(\varPhi_\theta)_j\rVert_{C^n(\mathbb R^{d+1})},
\]
using \eqref{EqU} gives
\begin{align}
\lVert(\varPhi_\theta)_j\rVert_{C^n(\mathbb R^{d+1})}
&\le
16^L(d+1)^{2n}
\Big(
e^2n^4W^3R^n
\lVert\sigma\rVert_{C^n(\mathbb R)}
\Big)^{nL},
\qquad
1\le j\le2d,
\label{Eq0C1}
\end{align}
which proves \eqref{EqC1}.\\

\noindent
Now, it remains to prove \eqref{EqCU} and \eqref{EqCW}. By the definitions
of $\pi_u$ and $\pi_w$, for every\\ $1\le j\le d$,
$(u_\theta)_j
=
(\varPhi_\theta)_j\big\rvert_{(0,T)\times\Omega_f}$
and
$(w_\theta)_j
=
(\varPhi_\theta)_{d+j}\big\rvert_{(0,T)\times\Omega_s}$.
Therefore,
$$\lVert(u_\theta)_j\rVert_{C^n((0,T)\times\Omega_f)}
\le
\lVert(\varPhi_\theta)_j\rVert_{C^n(\mathbb R^{d+1})}.$$
Applying \eqref{Eq0C1} with component index $j$ gives
\begin{align}
\lVert(u_\theta)_j\rVert_{C^n((0,T)\times\Omega_f)}
&\le
16^L(d+1)^{2n}
\Big(
e^2n^4W^3R^n
\lVert\sigma\rVert_{C^n(\mathbb R)}
\Big)^{nL},
\qquad
1\le j\le d.
\label{53B0}
\end{align}
Similarly, since
$\lVert(w_\theta)_j\rVert_{C^n((0,T)\times\Omega_s)}
\le
\lVert(\varPhi_\theta)_{d+j}\rVert_{C^n(\mathbb R^{d+1})},$ and $d+1\le d+j\le2d$, applying
\eqref{Eq0C1} with component index $d+j$ gives
\begin{align}
\lVert(w_\theta)_j\rVert_{C^n((0,T)\times\Omega_s)}
&\le
16^L(d+1)^{2n}
\Big(
e^2n^4W^3R^n
\lVert\sigma\rVert_{C^n(\mathbb R)}
\Big)^{nL},
\qquad
1\le j\le d.
\label{530K}
\end{align}
Therefore, combining \eqref{Eq0C1}, \eqref{53B0}, and \eqref{530K} completes
the proof of Lemma \ref{lemC1}.
\end{proof}
\noindent
Our next goal now is to give an affirmative answer to
Question~\textbf{[Q2]} raised in Section 1, and finalize our error analysis. To this end, we introduce the generalization and training error functions whose difference determines the $L^2$ error bounds. We begin with the generalization error, which quantifies
the residuals over the full space-time domains, boundaries, and interface:

\begin{definition}[Generalization Error]
\noindent \noindent
\noindent
For the coupled heat-wave FSI model \eqref{1.0}, the generalization error is defined via the PINN residuals
in \eqref{resi}-\eqref{resIC} as follows:
\begin{equation}
\begin{aligned}
E_G(\theta)^2
:&={}
\int_0^T\!\!\int_{\Omega_f}\big\lVert \mathcal{R}_u[u_\theta](t,x)\big\rVert^2_{\mathbb{R}^d}\,dx\,dt
+\int_0^T\!\!\int_{\Omega_s}\big\lVert \mathcal{R}_w[w_\theta](t,x)\big\rVert^2_{\mathbb{R}^d}\,dx\,dt
\\
&\quad
+\int_0^T\!\!\int_{\Gamma_f}\big\lVert \mathcal{R}_{BC}[u_\theta](t,x)\big\rVert_{\mathbb{R}^d}^2\,ds(x)\,dt
+\int_0^T\!\!\int_{\Gamma_s}\big\lVert \mathcal{R}_{\rm IF1}[u_\theta,w_\theta](t,x)\big\rVert_{\mathbb{R}^d}^2\,ds(x)\,dt
\\
&\quad
+\int_0^T\!\!\int_{\Gamma_s}\big\lVert \mathcal{R}_{\rm IF2}[u_\theta,w_\theta](t,x)\big\rVert^2_{\mathbb{R}^d}\,ds(x)\,dt
+\int_{\Omega_s}
\left(
\big\lVert \mathcal{R}_{IC1}[w_\theta](x)\big\rVert^2_{\mathbb{R}^d}
+
\big\lVert \nabla \mathcal{R}_{IC1}[w_\theta](x)\big\rVert^2_{\mathbb{R}^{d\times d}}
\right)\,dx
\\
&\quad
+\int_{\Omega_s}\big\lVert \mathcal{R}_{IC2}[w_\theta](x)\big\rVert^2_{\mathbb{R}^d}\,dx+\int_{\Omega_f}\big\lVert \mathcal{R}_{IC3}[u_\theta](x)\big\rVert^2_{\mathbb{R}^d}\,dx,
\label{eq:EG}
\end{aligned}
\end{equation}

\noindent
where the functional $E_G(\theta)$ measures how well $(w_\theta, w_{\theta t},  u_\theta)$ satisfies the interior PDEs, the boundary condition on $\Gamma_f$, the interface conditions on $\Gamma_s$, and the initial conditions, with
$\lVert\cdot\rVert_{\mathbb R^d}$ denoting the Euclidean norm.
\end{definition}
\noindent Since these residual integrals cannot be evaluated exactly in practice, they
are approximated by finite quadrature sums. In particular, they are approximated by numerical quadrature. Now, assume that \[
S=\big(S_u,S_w,S_{BC},S_{IF1},S_{IF2},S_{IC1},S_{IC2},S_{IC3}\big)
\]
denote the corresponding quadrature nodes and weights. Then, we define the training error function as follows:

\begin{definition}[Training Error]
The resulting training
error is defined as 
\begin{align}
E_T(\theta,S)^2
:={}&
\sum_{n=1}^{N_u} \omega^u_n
\,\bigl\lVert \mathcal{R}_u[u_\theta](t^u_n,x^u_n)\bigr\rVert_{\mathbb{R}^d}^2
+\sum_{n=1}^{N_w} \omega^w_n
\,\bigl\lVert \mathcal{R}_w[w_\theta](t^w_n,x^w_n)\bigr\rVert_{\mathbb{R}^d}^2
+\sum_{n=1}^{N_{BC}} \omega^{BC}_n
\,\bigl\lVert \mathcal{R}_{BC}[u_\theta](t^{BC}_n,y^{BC}_n)\bigr\rVert_{\mathbb{R}^d}^2
\nonumber\\
&\quad
+\sum_{n=1}^{N_{IF}} \omega^{IF}_n
\,\bigl\lVert \mathcal{R}_{\rm IF1}[u_\theta,w_\theta](t^{IF}_n,x^{IF}_n)\bigr\rVert_{\mathbb{R}^d}^2
+\sum_{n=1}^{N_{IF}} \omega^{IF}_n
\,\bigl\lVert \mathcal{R}_{\rm IF2}[u_\theta,w_\theta](t^{IF}_n,x^{IF}_n)\bigr\rVert_{\mathbb{R}^d}^2
\nonumber\\
&\quad
+\sum_{n=1}^{N_{IC}} \omega^{IC}_n
\left(
\bigl\lVert \mathcal{R}_{IC1}[w_\theta](x^{IC}_n)\bigr\rVert_{\mathbb{R}^d}^2
+
\bigl\lVert \nabla\mathcal{R}_{IC1}[w_\theta](x^{IC}_n)\bigr\rVert_{\mathbb{R}^{d\times d}}^2
\right)
+\sum_{n=1}^{N_{IC}} \omega^{IC}_n
\,\bigl\lVert \mathcal{R}_{IC2}[w_\theta](x^{IC}_n)\bigr\rVert_{\mathbb{R}^d}^2
\nonumber\\
&\quad
+\sum_{n=1}^{N_{IC}} \omega^{IC}_n
\,\bigl\lVert \mathcal{R}_{IC3}[u_\theta](x^{IC}_n)\bigr\rVert_{\mathbb{R}^d}^2,
\label{ET}
\end{align}

\noindent
where $N_u,N_w,N_{BC},N_{IF}$, and $N_{IC}$ denote the numbers of
quadrature nodes on the corresponding fluid interior, solid interior,
outer boundary, interface, and initial domains, respectively, with associated
weights $\omega_n^u,\omega_n^w,\omega_n^{BC},\omega_n^{IF}$, and
$\omega_n^{IC}$. 
\medskip

\noindent Moreover, the trained PINN is determined by minimizing the training error:
\[
\theta^*(S)\in\arg\min_{\theta\in\Theta}E_T(\theta,S)^2.
\]
\end{definition}

\noindent
In what follows, a decomposition of the training error in \eqref{ET} in terms of inner, boundary and initial training errors will often be used. Hence, we describe this decomposition here:
\begin{align}
E_T(\theta,S)^2
={}&
E_{PDE,T}(\theta,S_{PDE})^2
+E_{BC,T}(\theta,S_{BC})^2
+E_{IF,T}(\theta,S_{IF})^2
+E_{IC,T}(\theta,S_{IC})^2,
\label{56FR}
\end{align}
where $E_{PDE,T}$, $E_{BC,T}$, $E_{IF,T}$, and $E_{IC,T}$
denote the PDE, outer boundary, interface, and initial 
training errors, respectively. More precisely,
\begin{align}
&E_{PDE,T}(\theta,S_{PDE})^2
:={}
\sum_{n=1}^{N_u} \omega^u_n
\,\bigl\lVert \mathcal{R}_u[u_\theta](t^u_n,x^u_n)\bigr\rVert_{\mathbb{R}^d}^2
+
\sum_{n=1}^{N_w} \omega^w_n
\,\bigl\lVert \mathcal{R}_w[w_\theta](t^w_n,x^w_n)\bigr\rVert_{\mathbb{R}^d}^2,
\label{sde4}
\\
&E_{BC,T}(\theta,S_{BC})^2
:={}
\sum_{n=1}^{N_{BC}} \omega^{BC}_n
\,\bigl\lVert \mathcal{R}_{BC}[u_\theta](t^{BC}_n,y^{BC}_n)\bigr\rVert_{\mathbb{R}^d}^2,
\\
&E_{IF,T}(\theta,S_{IF})^2
:={}
\sum_{n=1}^{N_{IF}}\Big( \omega^{IF}_n
\,\bigl\lVert \mathcal{R}_{\rm IF1}[u_\theta,w_\theta](t^{IF}_n,x^{IF}_n)\bigr\rVert_{\mathbb{R}^d}^2
+
\,\bigl\lVert \mathcal{R}_{\rm IF2}[u_\theta,w_\theta](t^{IF}_n,x^{IF}_n)\bigr\rVert_{\mathbb{R}^d}^2\Big),
\\
&E_{IC,T}(\theta,S_{IC})^2
:={}
\sum_{n=1}^{N_{IC}} \omega^{IC}_n
\Big(
\bigl\lVert \mathcal{R}_{IC1}[w_\theta](x^{IC}_n)\bigr\rVert_{\mathbb{R}^d}^2
+
\bigl\lVert \nabla\mathcal{R}_{IC1}[w_\theta](x^{IC}_n)\bigr\rVert_{\mathbb{R}^{d\times d}}^2
\\
&\hspace{30mm}
+
\bigl\lVert \mathcal{R}_{IC2}[w_\theta](x^{IC}_n)\bigr\rVert_{\mathbb{R}^d}^2
+
\bigl\lVert \mathcal{R}_{IC3}[u_\theta](x^{IC}_n)\bigr\rVert_{\mathbb{R}^d}^2
\Big).
\label{50JHK}
\end{align}
For its numerical evaluation, we use the midpoint quadrature rules introduced
in Section~\ref{quad}.
In the coupled setting $\Omega=\Omega_f\cup\Omega_s$, we denote
the PDE, outer boundary, interface, and initial quadratures by
$Q_{M_{PDE}}^{PDE}$, $Q_{M_{BC}}^{BC}$, $Q_{M_{IF}}^{IF}$, and
$Q_{M_{IC}}^{IC}$, respectively, and their corresponding numbers
of quadrature points by $M_{PDE}$, $M_{BC}$, $M_{IF}$, and $M_{IC}$.\\

\noindent Hence, we can write \eqref{sde4}-\eqref{50JHK}  in a compact manner as follows:
\begin{align}
&E_{PDE,T}(\theta,S_{PDE})^2
=
Q_{M_{PDE}}^{PDE}
\Big[
\lvert\mathcal R_u\rvert^2
+
\lvert\mathcal R_w\rvert^2
\Big],
\label{3F1}
\\
&E_{BC,T}(\theta,S_{BC})^2
=
Q_{M_{BC}}^{BC}
\Big[
\lvert\mathcal R_{BC}\rvert^2
\Big],
\label{3F2}
\\
&E_{IF,T}(\theta,S_{IF})^2
=
Q_{M_{IF}}^{IF}
\Big[
\lvert\mathcal R_{IF1}\rvert^2
+
\lvert\mathcal R_{IF2}\rvert^2
\Big],
\label{3F3}
\\
&E_{IC,T}(\theta,S_{IC})^2
=
Q_{M_{IC}}^{IC}
\Big[
\lvert\mathcal R_{IC1}\rvert^2
+
\lvert\nabla\mathcal R_{IC1}\rvert^2
+
\lvert\mathcal R_{IC2}\rvert^2
+
\lvert\mathcal R_{IC3}\rvert^2
\Big].
\label{ET1}
\end{align}
\noindent
Next, we address Question~\textbf{[Q2]} raised in the introduction. Using the generalization error defined in \eqref{eq:EG}, the decomposition of the training error in \eqref{56FR} together with \eqref{3F1}-\eqref{ET1}, and applying the quadrature estimate \eqref{quad_err} to the corresponding residual terms, we obtain
\begin{align}
\bigl|E_G(\theta^*(S))^2-E_T(\theta^*(S),S)^2\bigr|
\le{}&
C_{\mathrm{IC}}M_{\mathrm{IC}}^{-1/d}
+
C_{\mathrm{PDE}}M_{\mathrm{PDE}}^{-1/(d+1)}
\nonumber\\
&+
C_{\mathrm{BC}}M_{\mathrm{BC}}^{-1/d}
+
C_{\mathrm{IF}}M_{\mathrm{IF}}^{-1/d},
\label{41EW}
\end{align}
which provides an affirmative answer to Question~\textbf{[Q2]}. We next derive an estimate for the total $L^2$-error of the trained PINN in terms of the training error and the number of quadrature points.
\medskip

\noindent
We now state the corresponding error estimate in the following theorem.

\begin{theorem}
\label{Th 3.3A}
Let $T>0$ and $d\in\{2,3\}$. Assume that the solution $(w,w_t,u)$ of the PDE system \eqref{1.0}
satisfies
$u\in C^4(\overline{\Omega}_f\times[0,T])$ and
$w\in C^4(\overline{\Omega}_s\times[0,T])$.
Let $(w_\theta,w_{\theta t},u_\theta)$ be a PINN solution with parameters
$\theta\in\Theta_{L,W,R}$, and define
$\hat u:=u_\theta-u$,  $\hat w:=w_\theta-w$, and $\hat{w}_t=w_{\theta t}-w_t$.
Then
\begin{align}
&
\lVert\hat u\rVert_{L^2((0,T)\times\Omega_f)}^2
+\lVert\partial_t\hat w\rVert_{L^2((0,T)\times\Omega_s)}^2
+\lVert\hat w\rVert_{L^2(0,T;H^1(\Omega_s))}^2
\nonumber\\
&\qquad=
O\!\left(
E_T(\theta,S)
+E_T(\theta,S)^2
+M_{IC}^{-1/d}
+M_{PDE}^{-1/(d+1)}
+M_{BC}^{-1/(2d)}
+M_{IF}^{-1/(2d)}
\right),
\label{eqdz}
\end{align}
where $E_T(\theta,S)$ is the training error defined in \eqref{ET} (or \eqref{56FR}), and
$M_{IC}$, $M_{PDE}$, $M_{BC}$, and $M_{IF}$ are the numbers of the quadrature points of initial, PDE, outer boundary and interface quadratures, respectively.
\end{theorem}
\begin{remark}
As required in Theorem \ref{Th 3.3A}, one needs $C^4(\overline{\Omega}_f\times[0,T])$ and
$C^4(\overline{\Omega}_s\times[0,T])$ higher regularity of the solutions $u$ and $w,$ to the PDE system \eqref{1.0}, respectively in order obtain the error estimates of the PINN. In this regard, Theorem \ref{thm2.1} will be the key departure point for us to proceed to analyze those bounds in the proof of Theorem \ref{Th 3.3A} below. Indeed, since our PDE model \eqref{1.0} is the simplified version of the coupled PDE system in \cite{AG} (Stokes equation is replaced by the heat equation), the same regularity result given in Theorem \ref{thm2.1} will hold for the solution to our coupled system. In particular, if one combines the pseudodifferential operator argumentation of \cite{AG} with a bootstrapping argument for sufficiently smooth initial data  then the solution $\{u,w, w_t\}$ to \eqref{1.0} satisfies higher regularity $w \in L^\infty(0,T;H^k(\Omega_s))$ and $u \in L^2(0,T;H^k(\Omega_f))$  for $k\geq3.$ We should also note that this bootstrapping argument for our coupled system \eqref{1.0}, combined with the linear semigroup theory, will yield the following: if  
$\{w_0, w_{01},u_0\}\in D(\mathbf{A}^k)$ then the solution $\{w,w_t,u\}\in \cap^{k}_{j=0}~C^{k-j}([0,\infty);D(\mathbf{A}^j)])$ (see \cite[Theorem 4.4.7]{kesevan}.)  
\end{remark}

\noindent\textit{\textbf{Proof of Theorem \ref{Th 3.3A}}} \\ 

\noindent Our proof is based on the error estimate \eqref{22E} whose RHS is given via the norms of residuals introduced in \eqref{resi}-\eqref{resIC}. In particular, we will estimate the each residual term present in the constant $C_{\mathcal{R}}$ by the training error and the corresponding number of the quadrature points. We give our proof in three parts. 

\medskip

\noindent
\textbf{\underline{Part I:}}
By the quadrature error estimate~\eqref{quad_err} we have that for any $g \in C^2(\Lambda)$,
\begin{align}
\int_\Lambda g
\le
Q_M^\Lambda[g]
+
C_\Lambda\lVert g\rVert_{C^2(\Lambda)}M^{-1/d}.
\label{GHD}
\end{align}
Hence, we first apply \eqref{GHD} to the  first term ($\mathcal R_{IC1}$) on the right hand side of \eqref{aslm} by setting
$$
g(x)
:=
\lvert\mathcal R_{IC1}(x)\rvert^2
+
\lvert\nabla \mathcal R_{IC1}(x)\rvert^2,
\qquad x\in\Omega_s,
$$
which implies
\begin{align}
\lVert\mathcal R_{IC1}\rVert_{H^1(\Omega_s)}^2
\le
Q_{M_{IC}}^{IC}
[
\lvert\mathcal R_{IC1}\rvert^2
+
\lvert\nabla \mathcal R_{IC1}\rvert^2
]
+
C_{IC1}M_{IC}^{-1/d}.
\label{56HT}
\end{align}
Using the same argument, for the second and third terms
($\mathcal R_{IC2}$ and $\mathcal R_{IC3}$)
on the right-hand side of \eqref{aslm},  we take
\begin{align}
\lVert\mathcal R_{IC2}\rVert_{L^2(\Omega_s)}^2
&=
\int_{\Omega_s}
\lvert\mathcal R_{IC2}\rvert^2\,dx
\nonumber\\
&\le
Q_{M_{IC}}^{IC}
[\lvert\mathcal R_{IC2}\rvert^2]
+
C_{IC2}M_{IC}^{-1/d},
\label{51s}
\\
\lVert\mathcal R_{IC3}\rVert_{L^2(\Omega_f)}^2
&=
\int_{\Omega_f}
\lvert\mathcal R_{IC3}\rvert^2\,dx
\nonumber\\
&\le
Q_{M_{IC}}^{IC}
[\lvert\mathcal R_{IC3}\rvert^2]
+
C_{IC3}M_{IC}^{-1/d}.
\label{52s}
\end{align}
Adding \eqref{56HT}, \eqref{51s}, and \eqref{52s}, we obtain
\begin{align}
&
\lVert\mathcal R_{IC1}\rVert_{H^1(\Omega_s)}^2
+
\lVert\mathcal R_{IC2}\rVert_{L^2(\Omega_s)}^2
+
\lVert\mathcal R_{IC3}\rVert_{L^2(\Omega_f)}^2
\nonumber\\
&\le
Q_{M_{IC}}^{IC}
[
\lvert\mathcal R_{IC1}\rvert^2
+
\lvert\nabla\mathcal R_{IC1}\rvert^2
+
\lvert\mathcal R_{IC2}\rvert^2
+
\lvert\mathcal R_{IC3}\rvert^2
]
\nonumber\\
&\quad
+
\Big(
C_{IC1}+C_{IC2}
\Big)M_{IC}^{-1/d}
+
C_{IC3}M_{IC}^{-1/d}.
\label{53B}
\end{align}
Setting
$
C_{IC,s}:=C_{IC1}+C_{IC2},
\quad
C_{IC,f}:=C_{IC3},
$
and recalling the definition of the initial training error 
$$
E_{IC,T}(\theta,S_{IC})^2
:=
Q_{M_{IC}}^{IC}
[
\lvert\mathcal R_{IC1}\rvert^2
+
\lvert\nabla\mathcal R_{IC1}\rvert^2
+
\lvert\mathcal R_{IC2}\rvert^2
+
\lvert\mathcal R_{IC3}\rvert^2
]
$$
(see also \eqref{ET1}), we have in   \eqref{53B} that
\begin{align}
&
\lVert\mathcal R_{IC1}\rVert_{H^1(\Omega_s)}^2
+
\lVert\mathcal R_{IC2}\rVert_{L^2(\Omega_s)}^2
+
\lVert\mathcal R_{IC3}\rVert_{L^2(\Omega_f)}^2
\nonumber\\
&\le
E_{IC,T}(\theta,S_{IC})^2
+
C_{IC,s}M_{IC}^{-1/d}
+
C_{IC,f}M_{IC}^{-1/d}.
\label{53E}
\end{align}
For the fourth and fifth terms ($\mathcal{R}_u$ and $\mathcal{R}_w$) on the right-hand side of \eqref{aslm}, using the same argument, we obtain
\begin{align}
\lVert\mathcal{R}_{u}\rVert_{L^2(\Omega_f\times[0,T])}^{2}
&=
\int_{\Omega_f\times[0,T]}\lvert\mathcal R_u\rvert^2\,dxdt \le
Q_{M_{PDE}}^{PDE}[\lvert\mathcal R_u\rvert^2]
+
C_uM_{PDE}^{-1/(d+1)},
\label{09s}
\\
\lVert\mathcal{R}_{w}\rVert_{L^2(\Omega_s\times[0,T])}^{2}
&=
\int_{\Omega_s\times[0,T]}\lvert\mathcal R_w\rvert^2\,dxdt\le
Q_{M_{PDE}}^{PDE}[\lvert\mathcal R_w\rvert^2]
+
C_wM_{PDE}^{-1/(d+1)},
\label{0913}
\end{align}
where \(d+1\) is the dimension of the space-time domains
\(\Omega_f\times[0,T]\) and \(\Omega_s\times[0,T]\). Again, recalling the definition in \eqref{3F1},
$$
E_{PDE,T}(\theta,S_{PDE})^2
:=
Q_{M_{PDE}}^{PDE}[\lvert\mathcal R_u\rvert^2]
+
Q_{M_{PDE}}^{PDE}[\lvert\mathcal R_w\rvert^2],
$$
the sum of \eqref{09s} and \eqref{0913} gives
\begin{align}
\label{57E}
\lVert\mathcal{R}_{u}\rVert_{L^2(\Omega_f\times[0,T])}^{2}
+
\lVert\mathcal{R}_{w}\rVert_{L^2(\Omega_s\times[0,T])}^{2}
\le
E_{PDE,T}(\theta,S_{PDE})^2
+
C_{PDE}M_{PDE}^{-1/(d+1)}.
\end{align}
For sixth term on the right hand of \eqref{aslm}, applying the same argument, we have 
\begin{align}
\lVert\mathcal{R}_{BC}\rVert_{L^2(\Gamma_f\times[0,T])}^{2}
&=
\int_{\Gamma_f\times[0,T]}\lvert\mathcal R_{BC}\rvert^2\,dSdt
\le
Q_{M_{BC}}^{BC}[\lvert\mathcal R_{BC}\rvert^2]
+
C_{BC}M_{BC}^{-1/d}.
\label{58E}
\end{align}
Since, by \eqref{3F2},
$
E_{BC,T}(\theta,S_{BC})^2
:=
Q_{M_{BC}}^{BC}[\lvert\mathcal R_{BC}\rvert^2],
$
it follows from \eqref{58E} that
\begin{align}
\label{59E}
\lVert\mathcal{R}_{BC}\rVert_{L^2(\Gamma_f\times[0,T])}
\le
\Big(
E_{BC,T}(\theta,S_{BC})^2
+
C_{BC}M_{BC}^{-1/d}
\Big)^{1/2}.
\end{align}
For the remaining terms in \eqref{aslm}, the same argument yields that
\begin{align}
\lVert\mathcal{R}_{IF1}\rVert_{L^2(\Gamma_s\times[0,T])}
&\le
\Big(
Q_{M_{IF}}^{IF}[\lvert\mathcal R_{IF1}\rvert^2]
+
C_{IF1}M_{IF}^{-1/d}
\Big)^{1/2},
\label{60E1}
\\
\lVert\mathcal{R}_{IF2}\rVert_{L^2(\Gamma_s\times[0,T])}
&\le
\Big(
Q_{M_{IF}}^{IF}[\lvert\mathcal R_{IF2}\rvert^2]
+
C_{IF2}M_{IF}^{-1/d}
\Big)^{1/2}.
\label{61E1}
\end{align}
Since \eqref{3F3} gives
\begin{align}
Q_{M_{IF}}^{IF}[\lvert\mathcal R_{IF1}\rvert^2],
\,
Q_{M_{IF}}^{IF}[\lvert\mathcal R_{IF2}\rvert^2]
\le
E_{IF,T}(\theta,S_{IF})^2,
\label{75RT}
\end{align}
it follows from \eqref{60E1}-\eqref{61E1} that
\begin{align}
\lVert\mathcal{R}_{IF1}\rVert_{L^2(\Gamma_s\times[0,T])}
&\le
\Big(
E_{IF,T}(\theta,S_{IF})^2
+
C_{IF1}M_{IF}^{-1/d}
\Big)^{1/2},
\label{60E2}
\\
\lVert\mathcal{R}_{IF2}\rVert_{L^2(\Gamma_s\times[0,T])}
&\le
\Big(
E_{IF,T}(\theta,S_{IF})^2
+
C_{IF2}M_{IF}^{-1/d}
\Big)^{1/2}.
\label{61E2}
\end{align}
Setting
$
C_{IF}:=C_{IF1}+C_{IF2},
$
and using \eqref{75RT}, it follows from \eqref{60E2}and \eqref{61E2} that
\begin{align}
\label{63E}
\lVert\mathcal{R}_{IF1}\rVert_{L^2(\Gamma_s\times[0,T])}^{2}
+
\lVert\mathcal{R}_{IF2}\rVert_{L^2(\Gamma_s\times[0,T])}^{2}
\le
E_{IF,T}(\theta,S_{IF})^2
+
C_{IF}M_{IF}^{-1/d}.
\end{align}
Taking into account
\eqref{53E}, \eqref{57E}, \eqref{59E},        \eqref{60E2},\eqref{61E2}, and \eqref{63E} in \eqref{aslm}, we get
\begin{align}
&C_{\mathcal{R}}=\Big(\lVert\mathcal{R}_{IC1}\rVert_{H^1(\Omega_s)}^2
+\lVert\mathcal{R}_{IC2}\rVert_{L^2(\Omega_s)}^2
+\lVert\mathcal{R}_{IC3}\rVert_{L^2(\Omega_f)}^2
+\lVert\mathcal{R}_{u}\rVert_{L^2(\Omega_f\times[0,T])}^{2}
+\lVert\mathcal{R}_{w}\rVert_{L^2(\Omega_s\times[0,T])}^{2}
\nonumber\\
&
+C\lVert\mathcal{R}_{BC}\rVert_{L^2(\Gamma_f\times[0,T])}
+C\lVert\mathcal{R}_{IF1}\rVert_{L^2(\Gamma_s\times[0,T])}
+C\lVert\mathcal{R}_{IF2}\rVert_{L^2(\Gamma_s\times[0,T])}
\nonumber\\
&
+\lVert\mathcal{R}_{IF1}\rVert_{L^2(\Gamma_s\times[0,T])}^{2}
+\lVert\mathcal{R}_{IF2}\rVert_{L^2(\Gamma_s\times[0,T])}^{2}\Big)
\nonumber\\
& \hspace{48mm}\leq \Big(
E_{IC,T}(\theta,S_{IC})^2
+
C_{IC,s}M_{IC}^{-1/d}
+
C_{IC,f}M_{IC}^{-1/d}
\nonumber\\
&\hspace{48mm}+
E_{PDE,T}(\theta,S_{PDE})^2
+
C_{PDE}M_{PDE}^{-1/(d+1)}
\nonumber\\
&\hspace{48mm}+C
\Big(
E_{BC,T}(\theta,S_{BC})^2
+
C_{BC}M_{BC}^{-1/d}
\Big)^{1/2}
\nonumber\\
&\hspace{48mm}+
C\Big(
E_{IF,T}(\theta,S_{IF})^2
+
C_{IF}M_{IF}^{-1/d}
\Big)^{1/2}
+
E_{IF,T}(\theta,S_{IF})^2
+
C_{IF}M_{IF}^{-1/d}\Big)
\nonumber\\
&\hspace{48mm}=\bar C
\label{dfo}
\end{align}
Hence, considering \eqref{dfo} in \eqref{22E}, we obtain
\begin{align}
\lVert\hat u\rVert_{L^2((0,T)\times\Omega_f)}^2
+\lVert\partial_t\hat w\rVert_{L^2((0,T)\times\Omega_s)}^2
+\lVert\hat w\rVert_{L^2(0,T;H^1(\Omega_s))}^2
\le
\bar C\bigl(T+T^2e^T\bigr). \label{main}
\end{align}
\medskip

\noindent \textbf{\noindent \underline{Part II.}} Having obtained a bound for the residuals in \textbf{Part I} \eqref{dfo} in terms of the training error functions, we now turn our attention to the constants
$C_{PDE},C_{BC},C_{IF},C_{IC,s}$ and $C_{IC,f}$ in
\eqref{dfo}. Our goal in this part is to estimate those constants in terms of the neural network parameters
$L$, $W$, $R$, and the activation function $\sigma$, which is defined in
Definition~\ref{NN}.
\medskip

\noindent First, for $n\in\mathbb{N}$, Lemma~\ref{lemC1} gives the following
estimates for the PINN approximations $u_\theta$ and $w_\theta$:
\begin{align}
\label{eq3.33}
&\lVert {(u_{\theta})}_j \rVert_{C^n}
\le
16^L (d+1)^{2n}
(
e^{2} n^4 W^3 R^n \lVert \sigma\rVert_{C^n}
)^{nL},\\
&\lVert {(w_{\theta})}_j \rVert_{C^n}
\le
16^L (d+1)^{2n}
(
e^{2} n^4 W^3 R^n \lVert \sigma\rVert_{C^n}
)^{nL},
\label{105K}
\end{align}
for all $j$.
Now, we apply Lemma~\ref{lemP} in Appendix to the residual terms as follows:
\medskip

\noindent
Using the PDE residuals in $\eqref{1000}_{1}$ and $\eqref{1000}_{2}$, we have
\begin{align}
\mathcal R_u
&=\hat u_t-\Delta\hat u
=\partial_tu_\theta-\Delta u_\theta
-\bigl(u_t-\Delta u\bigr)
=\partial_tu_\theta-\Delta u_\theta
\label{30P},\\
\mathcal R_w
&=\hat w_{tt}-\Delta\hat w+\hat w
=\partial_{tt}w_\theta-\Delta w_\theta+w_\theta
-\bigl(w_{tt}-\Delta w+w\bigr)
\nonumber\\
&=\partial_{tt}w_\theta-\Delta w_\theta+w_\theta.
\label{30B}
\end{align}
Hence, by Lemma \eqref{lemP}, \eqref{quad_err}, \eqref{eq3.33}, \eqref{105K}, \eqref{30P}, and \eqref{30B}, we get
\begin{align}
C_{PDE}
&\lesssim
\lVert \mathcal R_u^2\rVert_{C^2((0,T)\times\Omega_f)}
+
\lVert \mathcal R_w^2\rVert_{C^2((0,T)\times\Omega_s)}
\nonumber\\
&\lesssim
\lVert u_\theta\rVert_{C^4(\overline{\Omega}_f\times[0,T])}^2
+
\lVert w_\theta\rVert_{C^4(\overline{\Omega}_s\times[0,T])}^2
\nonumber\\
&\lesssim
16^L(d+1)^8
\left(
e^2 4^4W^3R^4\lVert\sigma\rVert_{C^4}
\right)^{4L}.
\label{CE0}
\end{align}
Since the
boundary residual  
$
\mathcal R_{BC}
=\hat u|_{\Gamma_f}
=(u_\theta-u)|_{\Gamma_f}
=u_\theta|_{\Gamma_f}
$ in $\eqref{1000}_{3}$,
we have that
$$
\lVert\mathcal R_{BC}\rVert_{C^2((0,T)\times\Gamma_f)}
\lesssim
\lVert u_\theta\rVert_{C^2(\overline{\Omega}_f\times[0,T])}
$$
and hence, combining this estimate  with \eqref{quad_err}, \eqref{eq3.33}, and Lemma \eqref{lemP}, we get
\begin{align}
C_{BC}
&\lesssim
\lVert\mathcal R_{BC}^2\rVert_{C^2((0,T)\times\Gamma_f)}
\nonumber\\
&\lesssim
16^L(d+1)^4
\left(
e^2 2^4W^3R^2\lVert\sigma\rVert_{C^2}
\right)^{2L}.
\label{CBC}
\end{align}
Recalling from $\eqref{1000}_{4}$ and $\eqref{1000}_{5}$ that the interface residuals are given by
\begin{equation}
\begin{aligned}
&\mathcal R_{IF1}=\hat u-\hat w_t
=(u_\theta-w_{\theta,t})-(u-w_t)
=u_\theta-w_{\theta,t},
\nonumber\\
&\mathcal R_{IF2}
=\frac{\partial\hat u}{\partial\nu}
-\frac{\partial\hat w}{\partial\nu}
=
\left(
\frac{\partial u_\theta}{\partial\nu}
-\frac{\partial w_\theta}{\partial\nu}
\right)
-
\left(
\frac{\partial u}{\partial\nu}
-\frac{\partial w}{\partial\nu}
\right)=
\frac{\partial u_\theta}{\partial\nu}
-\frac{\partial w_\theta}{\partial\nu},
\end{aligned}
\end{equation}
 and applying the triangle inequality,  we also have 
\begin{align}
\lVert\mathcal R_{IF1}\rVert_{C^2((0,T)\times\Gamma_s)}
&\lesssim
\lVert u_\theta\rVert_{C^2(\overline{\Omega}_f\times[0,T])}
+
\lVert w_\theta\rVert_{C^3(\overline{\Omega}_s\times[0,T])},
\label{81B}
\\
\lVert\mathcal R_{IF2}\rVert_{C^2((0,T)\times\Gamma_s)}
&\lesssim
\lVert u_\theta\rVert_{C^3(\overline{\Omega}_f\times[0,T])}
+
\lVert w_\theta\rVert_{C^3(\overline{\Omega}_s\times[0,T])}.
\label{34C}
\end{align}
Setting $C_{IF}=C_{IF1}+C_{IF2}$ and using
\eqref{eq3.33}, \eqref{105K}, \eqref{81B}, and \eqref{34C}, we take
\begin{align}
C_{IF}
&\lesssim
\lVert\mathcal R_{IF1}^2\rVert_{C^2((0,T)\times\Gamma_s)}
+
\lVert\mathcal R_{IF2}^2\rVert_{C^2((0,T)\times\Gamma_s)}
\nonumber\\
&\lesssim
\lVert u_\theta\rVert_{C^3(\overline{\Omega}_f\times[0,T])}^2
+
\lVert w_\theta\rVert_{C^3(\overline{\Omega}_s\times[0,T])}^2
\nonumber\\
&\lesssim
16^L(d+1)^6
\left(
e^2 3^4W^3R^3\lVert\sigma\rVert_{C^3}
\right)^{3L}.
\label{CIF}
\end{align}
Finally, using the initial residuals given in $\eqref{1000}_{6}$
and applying the triangle inequality, we obtain
\begin{align}
\lVert\mathcal R_{IC1}\rVert_{C^3(\Omega_s)}
&\lesssim
\lVert w_\theta\rVert_{C^3(\overline{\Omega}_s\times[0,T])}
+
\lVert w\rVert_{C^3(\overline{\Omega}_s\times[0,T])},
\label{85A1}\\
\lVert\mathcal R_{IC2}\rVert_{C^2(\Omega_s)}
&\lesssim
\lVert w_\theta\rVert_{C^3(\overline{\Omega}_s\times[0,T])}
+
\lVert w\rVert_{C^3(\overline{\Omega}_s\times[0,T])},
\label{85A2}\\
\lVert\mathcal R_{IC3}\rVert_{C^2(\Omega_f)}
&\lesssim
\lVert u_\theta\rVert_{C^2(\overline{\Omega}_f\times[0,T])}
+
\lVert u\rVert_{C^2(\overline{\Omega}_f\times[0,T])}.
\label{36C}
\end{align}
Therefore, from \eqref{quad_err},  \eqref{105K}, Lemma \eqref{lemP}, \eqref{85A1}, and \eqref{85A2},
we obtain
\begin{align}
C_{IC,s}
&\lesssim
\left\lVert 
\lvert\mathcal R_{IC1}\rvert^2
+
\lvert\nabla\mathcal R_{IC1}\rvert^2
\right\rVert_{C^2(\Omega_s)}
+
\lVert\mathcal R_{IC2}^2\rVert_{C^2(\Omega_s)}
\nonumber\\
&\lesssim
\lVert\mathcal R_{IC1}\rVert_{C^3(\Omega_s)}^2
+
\lVert\mathcal R_{IC2}\rVert_{C^2(\Omega_s)}^2
\nonumber\\
&\lesssim
\lVert w_\theta\rVert_{C^3(\overline{\Omega}_s\times[0,T])}^2
+
\lVert w\rVert_{C^3(\overline{\Omega}_s\times[0,T])}^2
\nonumber\\
&\lesssim
16^L(d+1)^6
\left(
e^2 3^4W^3R^3\lVert\sigma\rVert_{C^3}
\right)^{3L}
+
2\lVert w\rVert_{C^3(\overline{\Omega}_s\times[0,T])}^2.
\label{CICs}
\end{align}
Similarly, from \eqref{quad_err}, \eqref{eq3.33}, Lemma \eqref{lemP}, and \eqref{36C}, we have
\begin{align}
C_{IC,f}
&\lesssim
\lVert\mathcal R_{IC3}^2\rVert_{C^2(\Omega_f)}
\lesssim
\lVert u_\theta\rVert_{C^2(\overline{\Omega}_f\times[0,T])}^2
+
\lVert u\rVert_{C^2(\overline{\Omega}_f\times[0,T])}^2
\nonumber\\
&\lesssim
16^L(d+1)^4
\left(
e^2 2^4W^3R^2\lVert\sigma\rVert_{C^2}
\right)^{2L}
+
2\lVert u\rVert_{C^2(\overline{\Omega}_f\times[0,T])}^2.
\label{CICf}
\end{align}

\medskip
\noindent
\textbf{\underline{Part III:}}
In the last part of the proof, we recall the constant $\bar C$ in \eqref{dfo} and the estimates \eqref{CE0}, \eqref{CBC}, \eqref{CIF}, \eqref{CICs}, and \eqref{CICf}  obtained in \textbf{Part II} for the constants $C_{PDE},C_{BC},C_{IF},C_{IC,s}, \textrm{and}\, C_{IC,f}$. \\ \\
\noindent Since 
$
(a+b)^{1/2}\leq a^{1/2}+b^{1/2},
$
for any $a,b>0$, we have for the terms on the RHS of \eqref{dfo} that
\begin{equation}
\begin{aligned}
\Big(
E_{BC,T}(\theta,S_{BC})^2
+
C_{BC}M_{BC}^{-1/d}
\Big)^{1/2}
&\le
C\Big(
E_{BC,T}(\theta,S_{BC})
+
M_{BC}^{-1/(2d)}
\Big),
\\
\Big(
E_{IF,T}(\theta,S_{IF})^2
+
C_{IF}M_{IF}^{-1/d}
\Big)^{1/2}
&\le
C\Big(
E_{IF,T}(\theta,S_{IF})
+
M_{IF}^{-1/(2d)}
\Big).
\end{aligned}
\label{96RT}
\end{equation}
Also, since for every fixed $T>0~$, $(T+T^2e^T)$ is a constant, using
\eqref{96RT} in \eqref{dfo}, we get
\begin{align}
\bar C\big(T+T^2e^T\big)=
O\Big(
&E_{IC,T}(\theta,S_{IC})^2
+
E_{PDE,T}(\theta,S_{PDE})^2
+
E_{BC,T}(\theta,S_{BC})\nonumber\\
&+
E_{IF,T}(\theta,S_{IF})
+
E_{IF,T}(\theta,S_{IF})^2
+
M_{IC}^{-1/d}
+
M_{PDE}^{-1/(d+1)}\nonumber\\
&+
M_{BC}^{-1/(2d)}
+
M_{IF}^{-1/(2d)}
+
M_{IF}^{-1/d}
\Big).
\label{91S}
\end{align}
Finally, recalling the training error in \eqref{56FR} 
\begin{align*}
E_T(\theta,S)^2
&=
E_{PDE,T}(\theta,S_{\mathrm{PDE}})^2
+
E_{BC,T}(\theta,S_{BC})^2
+
E_{IF,T}(\theta,S_{IF})^2
+
E_{IC,T}(\theta,S_{IC})^2
\end{align*}
and observing that for $M_{IF}\ge1~$,
$
M_{IF}^{-1/d}\le M_{IF}^{-1/(2d)},
$ \eqref{91S} and \eqref{main} yields that 
\begin{align}
\bar C\big(T+T^2e^T\big)
=
O\Big(
E_T(\theta,S)
+
E_T(\theta,S)^2
+
M_{IC}^{-1/d}
+
M_{PDE}^{-1/(d+1)}
+
M_{BC}^{-1/(2d)}
+
M_{IF}^{-1/(2d)}
\Big),
\end{align}
which gives \eqref{eqdz} and completes the proof of Theorem~\ref{Th 3.3A}.

\medskip
 \section{Numerical Experiment}
\label{NE}
\noindent 
In this section we present numerical experiments validating 
the error estimates of Theorem~\ref{Th 3.3A} for the PINN 
approximation of the coupled heat-wave FSI system \eqref{1.0}.


\medskip
\noindent\textbf{Geometry and manufactured solution.}
For $r(x,y) = \sqrt{x^2+y^2}$ and consider the circle-in-circle 
configuration with $R_s = 0.5$ and $R_f = 1$:
$$
\Omega_s = \{(x,y)\in\mathbb{R}^2 : r < R_s\},
\qquad
\Omega_f = \{(x,y)\in\mathbb{R}^2 : R_s < r < R_f\},
$$
with interface $\Gamma_s := \{r = R_s\}$ and outer boundary 
$\Gamma_f := \{r = R_f\}$.

\noindent Set $\phi(r) := 1 - r^2$. We use the exact solutions $ 
\bigl((w_1,w_2),(w_{1t},w_{2t}),(u_1,u_2)\bigr)$
defined as follows:
\begin{align*}
u_1(x,y,t) &= e^t\phi(r),
&\qquad u_2(x,y,t) &= 2e^t\phi(r),\\
w_1(x,y,t) &= e^t\phi(r)+1,
&\qquad w_2(x,y,t) &= 2e^t\phi(r)+2,\\
 w_{1t}(x,y,t) &= e^t\phi(r),
&\qquad  w_{2t}(x,y,t) &= 2e^t\phi(r).
\end{align*}
The corresponding source terms are
$$
f_1=\bigl((5-r^2)e^t,\,(10-2r^2)e^t\bigr),
\qquad
f_2=\bigl((6-2r^2)e^t+1,\,(12-4r^2)e^t+2\bigr),
$$
for the fluid and solid equations, respectively.
\medskip

\noindent\textbf{Network architecture and training.}
A single fully-connected neural network 
$\varPhi_\theta:\mathbb{R}^3\to\mathbb{R}^4$
approximates the coupled fields, mapping $(x,y,t)$ to 
$((u_{\theta\,1},u_{\theta\,2}),(w_{\theta\,1},w_{\theta\,2}))$, while $(w_{1\,\theta\,t},w_{2\,\theta\,t})$ is obtained by 
automatic differentiation with respect to $t$. The architecture uses two hidden layers 
of width $W$ with $\tanh$ activation, giving $L=3$ in the notation of
Definition~\ref{NN}, and output dimension $l_L=2d=4$.
\medskip

\begin{figure}[h]
\centering
\begin{minipage}{0.48\linewidth}
    \centering
    \includegraphics[width=\linewidth]{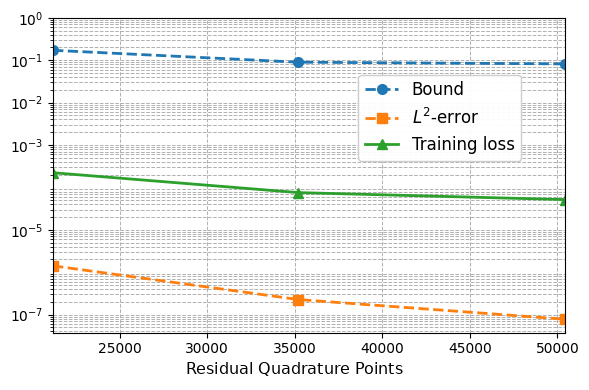}
\end{minipage}
\hfill
\begin{minipage}{0.48\linewidth}
    \centering
    \includegraphics[width=\linewidth]{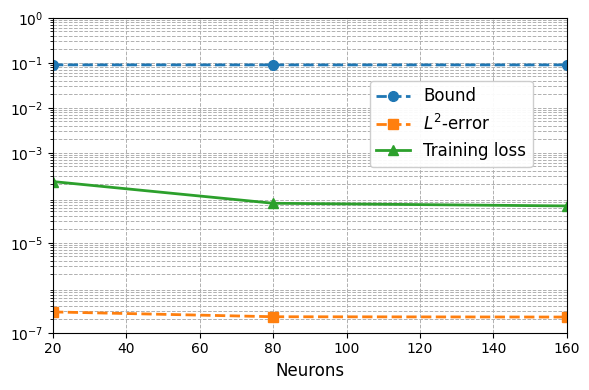}
\end{minipage}
\caption{Total $L^2$-error, training error $E_T(\theta,S)$, 
and theoretical bound \eqref{eqdz} as functions of 
(left) number of residual quadrature points $M$ with fixed 
width $W=80$, and (right) network width $W$ with fixed 
$M = 35{,}200$.}
\label{twoim}
\end{figure}

\begin{figure}[!h]
    \centering
    \begin{minipage}{0.33\linewidth}
        \centering
        Exact\\[2pt]
        \includegraphics[width=\linewidth, clip, trim = 0 0 0 5.1cm]{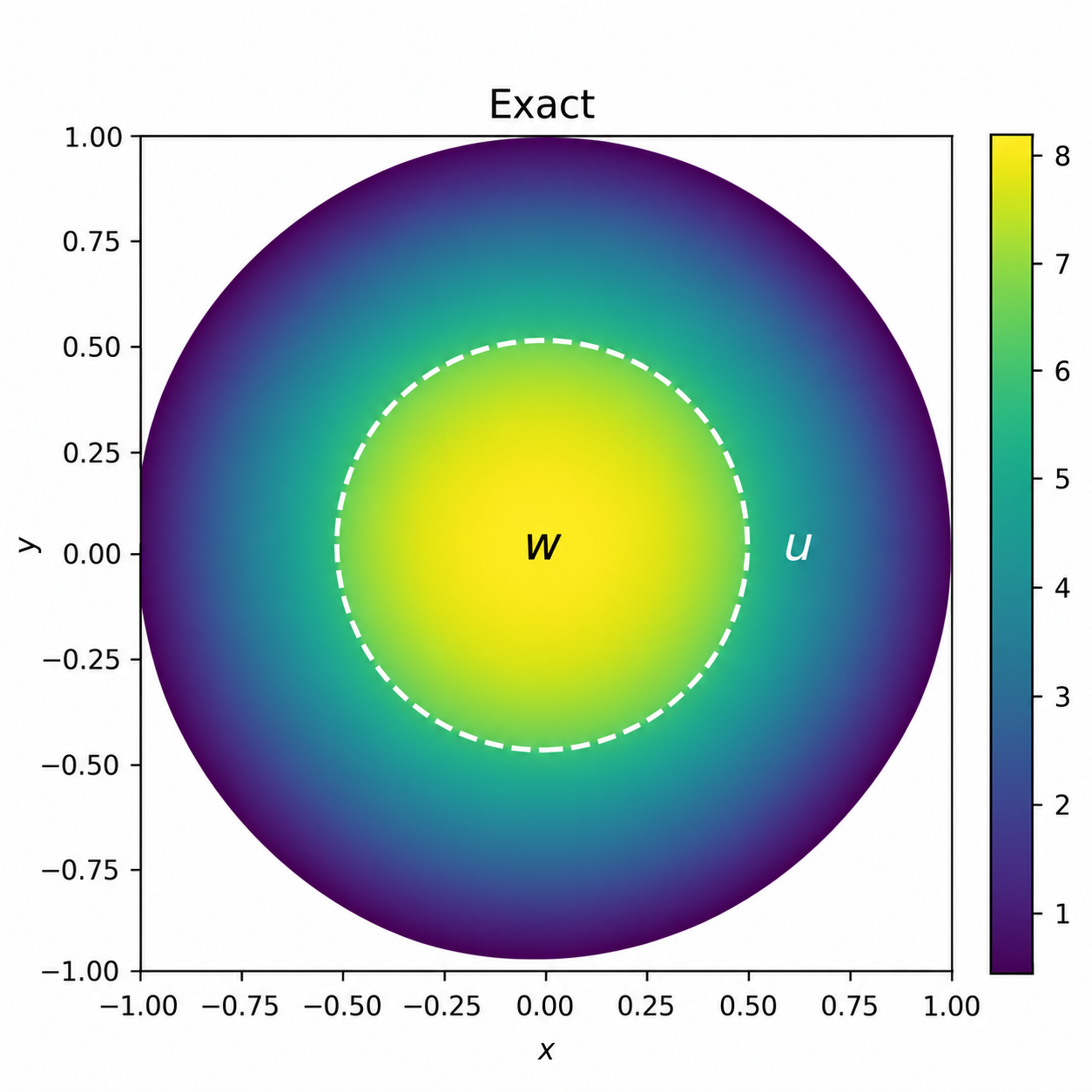}
    \end{minipage}%
    \hfill
    \begin{minipage}{0.33\linewidth}
        \centering
        PINN approximation\\[2pt]
        \includegraphics[width=\linewidth, clip, trim = 0 0 0 5cm]{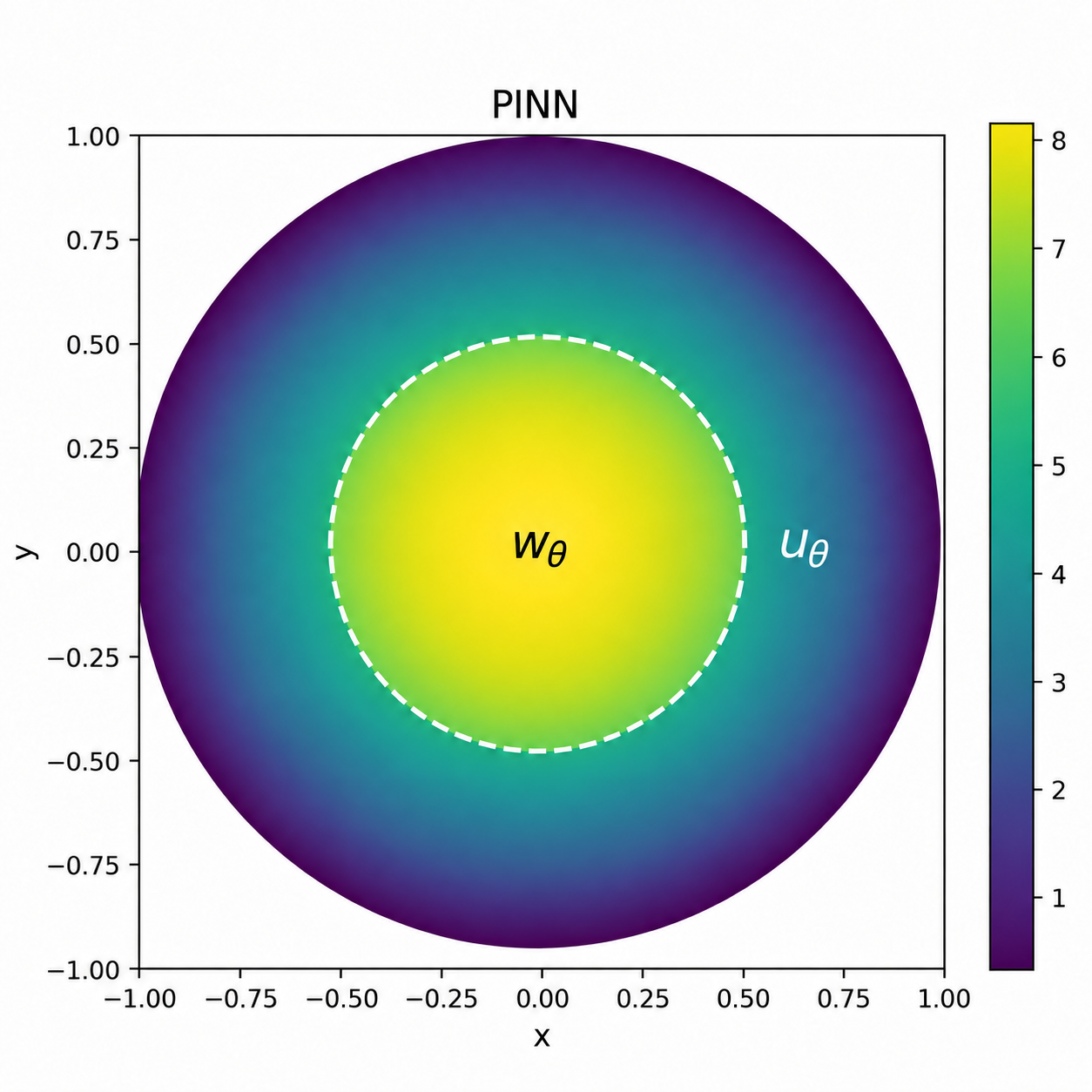}
    \end{minipage}%
    \hfill
    \begin{minipage}{0.33\linewidth}
        \centering
        Absolute Error\\[2pt]
        \includegraphics[width=\linewidth, clip, trim = 0 0 0 5.6cm]{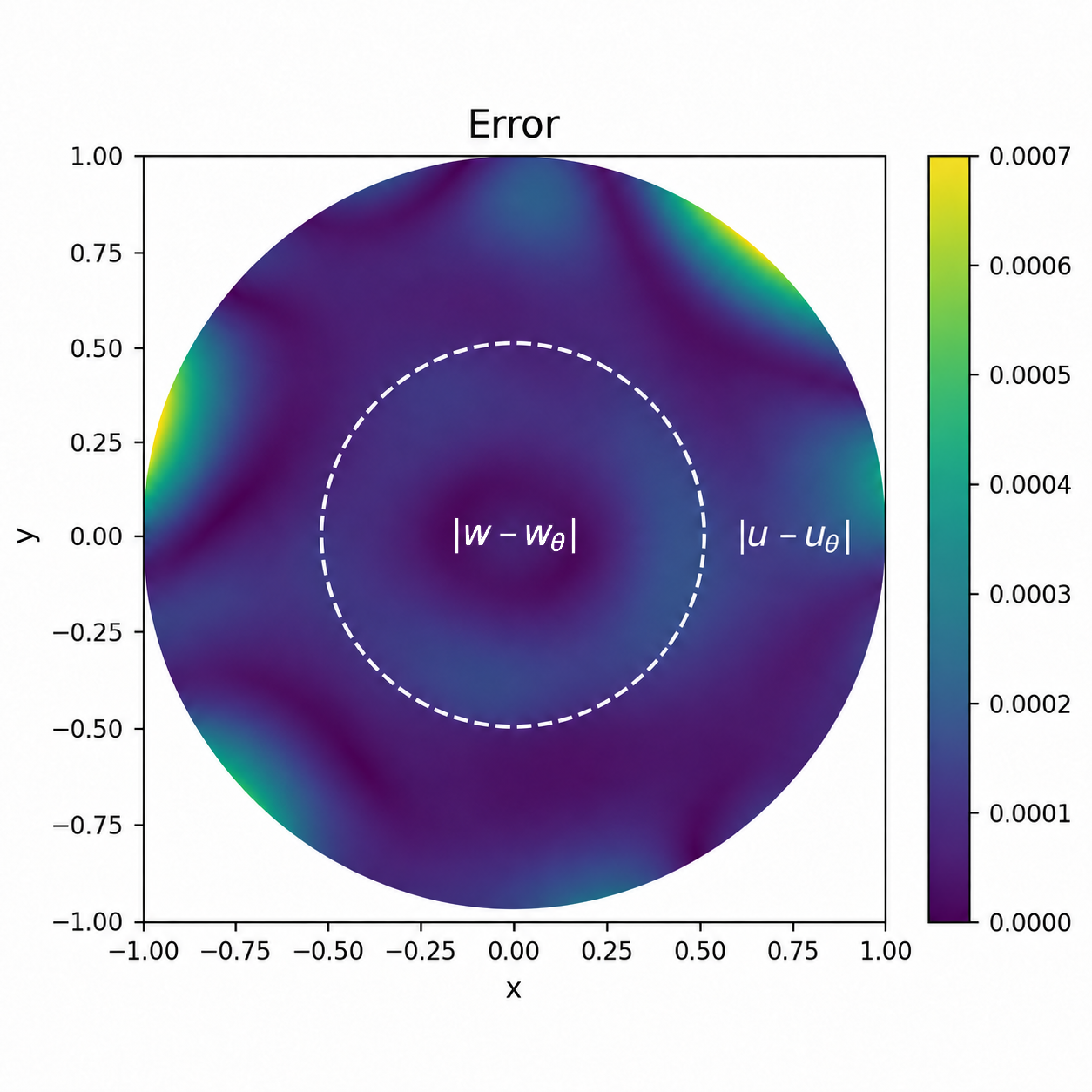}
    \end{minipage}
    \centering
    \begin{minipage}{0.33\linewidth}
        \centering
        \includegraphics[width=\linewidth, clip, trim = 0 0 0 5.1cm]{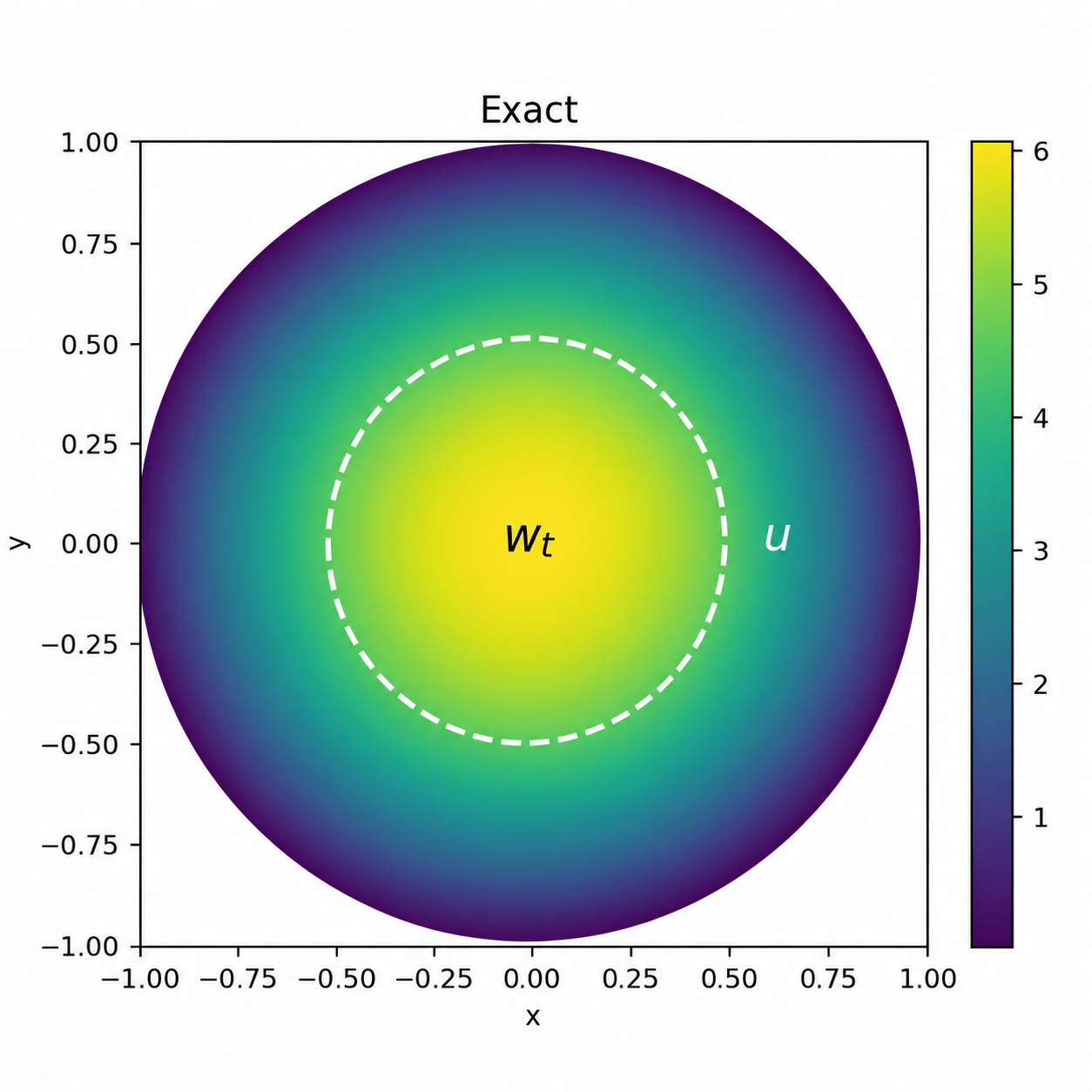}
    \end{minipage}%
    \hfill
    \begin{minipage}{0.33\linewidth}
        \centering
        \includegraphics[width=\linewidth, clip, trim = 0 0 0 5.2cm]{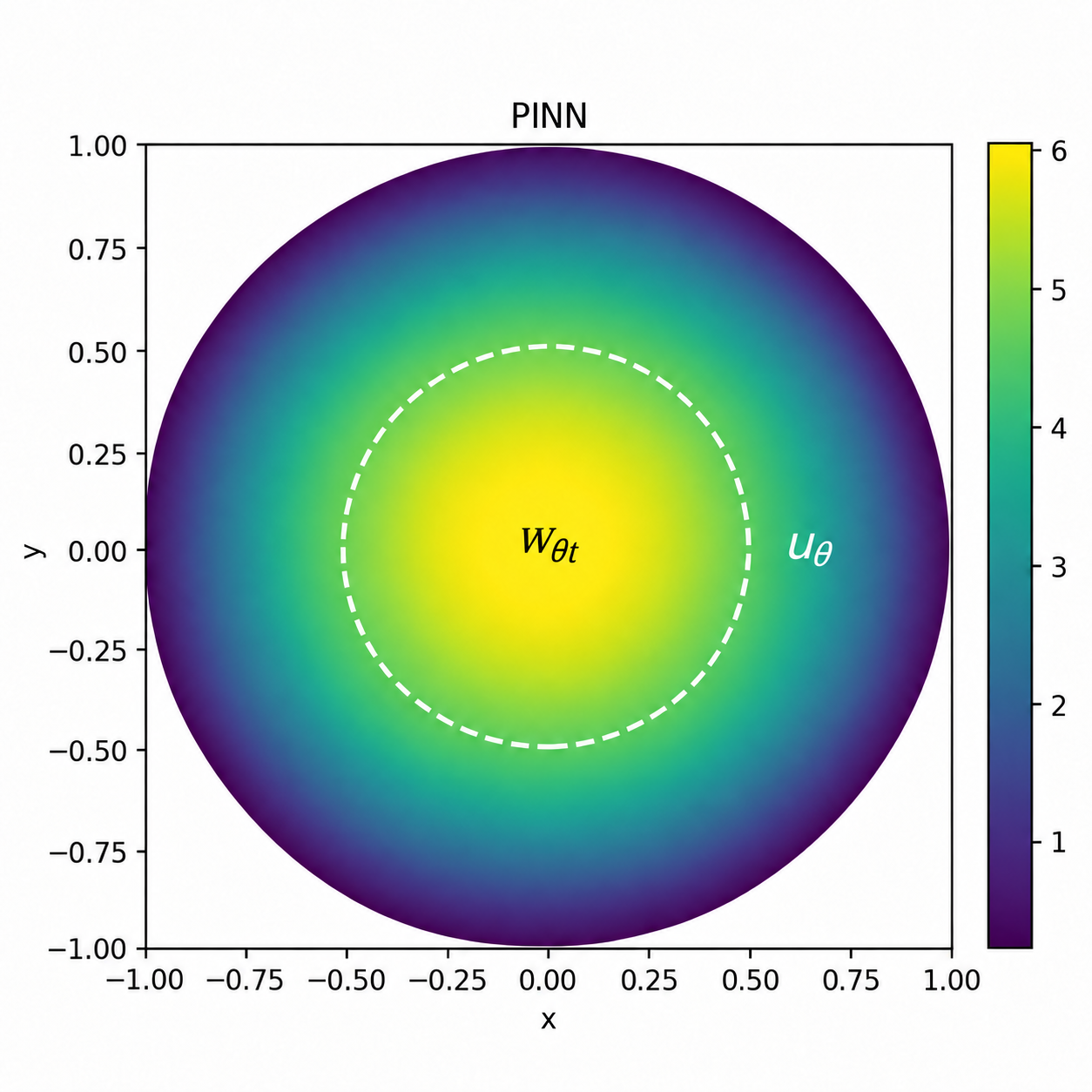}
    \end{minipage}%
    \hfill
    \begin{minipage}{0.33\linewidth}
        \centering
        \includegraphics[width=\linewidth, clip, trim = 0 0 0 5.8cm]{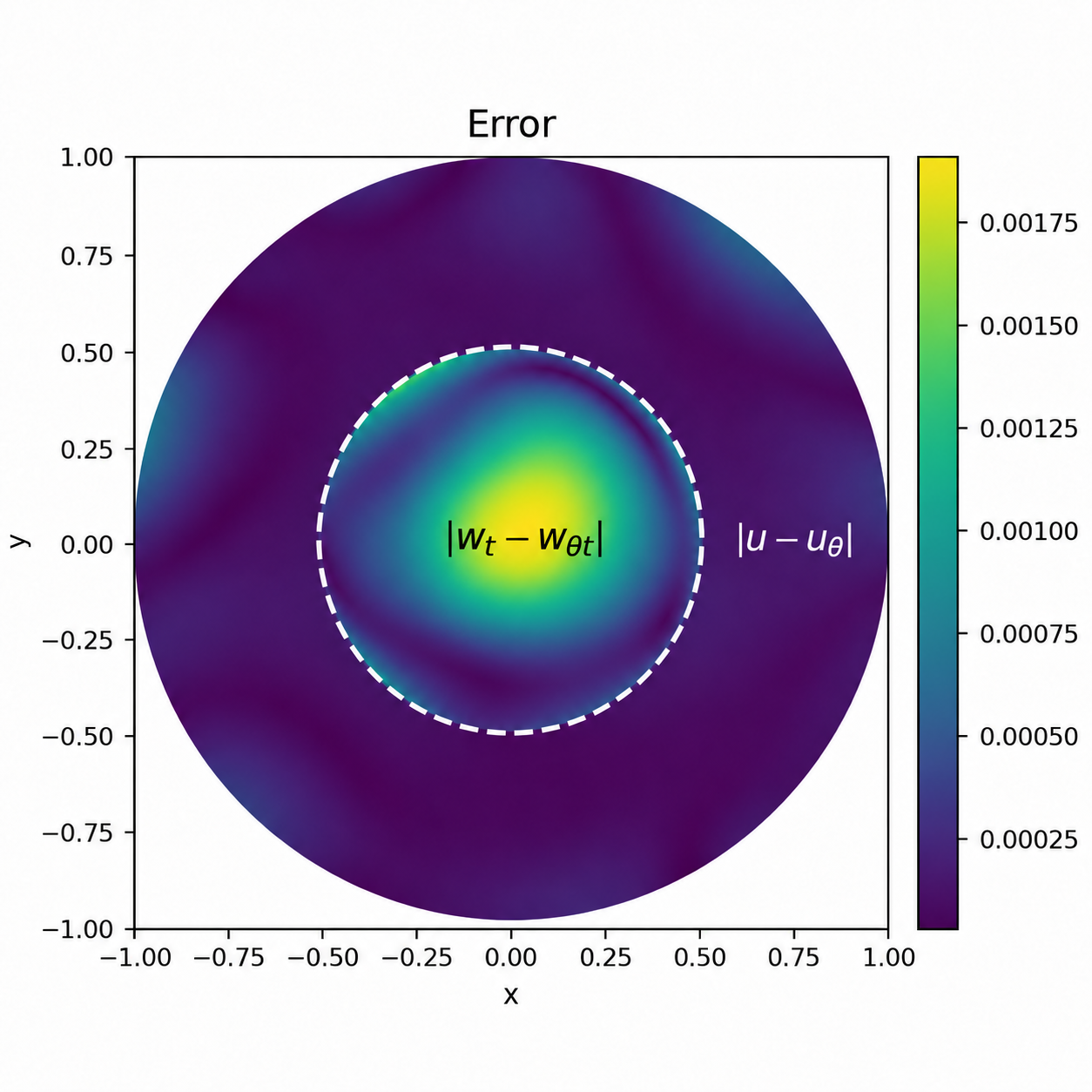}
    \end{minipage}
    \caption{Manufactured solution (left), PINN approximation (middle), and absolute pointwise error (right) at $t=1$, for $M=21{,}200$ quadrature points and $W=80$ neurons per hidden layer. Top row: $u$ and $w$; bottom row: $u$ and $w_t$.} 
    \label{fig 3}
\end{figure}
\medskip
\noindent 
Training is performed in two stages. In the first stage, the Adam optimizer is run for $20{,}000$ epochs with an initial learning rate of $5\times10^{-4}$. The learning rate is reduced by half when the training loss plateaus, to a minimum of $6.25\times10^{-5}$; and gradient clipping with threshold $1.0$ is employed throughout. In the second stage, the L-BFGS optimizer is run for six cycles of $1{,}000$ iterations each to further reduce the PDE, boundary, and interface residuals.
\medskip

\noindent\textbf{Experiment 1: Effect of quadrature points.}
We fix the network width at $W = 80$ and vary the number of 
residual quadrature points $M \in \{21{,}200,\,35{,}200,\,50{,}450\}$.  According to the estimate \eqref{eqdz}, the bound on the total 
error contains the quadrature-error terms
$C_{IC}M_{IC}^{-1/d}$,
$C_{\mathrm{PDE}}M_{\mathrm{PDE}}^{-1/(d+1)}$,
$C_{BC}M_{BC}^{-1/2d}$, and
$C_{IF}M_{IF}^{-1/d}$, which decrease as the corresponding
numbers of quadrature points increase.

The left panel of 
Figure~\ref{twoim} reports the total error 
$\displaystyle \lVert (w, w_t, u) - (w_{\theta},w_{\theta t}, u_{\theta})\rVert_{L^2},$ the training error 
$E_T(\theta,S)$, and the theoretical bound \eqref{eqdz} for 
each value of $M$. The results confirm that increasing the 
number of quadrature points reduces both the training loss 
and the $L^2$-error, with the most significant reduction 
occurring at the initial increase from $21{,}200$ to 
$35{,}200$ points. The theoretical bound decreases consistently with 
$M$, in agreement with \eqref{eqdz}.
\medskip

\noindent
Figure~\ref{fig 3} displays the manufactured solution, PINN
approximation, and pointwise absolute error for $(w,w_t,u)$ based on
training with $M=21{,}200$ residual quadrature points.

\medskip
\noindent\textbf{Experiment 2: Effect of network width.}
We fix $M = 35{,}200$ quadrature points and two hidden layers, and vary the width $W \in \{20,\,80,\,160\}$. In \eqref{eqdz}, the width enters only through the constants $C_{\mathrm{PDE}}$, $C_{BC}$, $C_{IF}$, and $C_{IC}$, which by Lemma~\ref{lemC1} grow with $W$; these are worst-case bounds over the whole parameter set $\Theta_{L,W,R}$ and need not be sharp for a trained network. At fixed $M$, the quadrature terms in \eqref{eqdz} do not depend on $W$, so a wider network can improve the estimate only through the training error. The right panel of Figure~\ref{twoim} reports the total error, the training loss, and the bound for each width; the plotted bound is the right-hand side of \eqref{eqdz}. The training loss decreases from $W=20$ to $W=80$ and changes little between $W=80$ and $W=160$, and the $L^2$-error varies only slightly across the three widths. The left panel of Figure~\ref{twoim}, the measured error stays below the plotted bound in every run, decreases as $M$ grows, and is insensitive to $W$ at fixed $M$, consistent with the dependence on $M$ and on the training error in \eqref{eqdz}. 

\section{Appendix}
\vspace{0.5cm}
\noindent We recall the following lemma for continuously differentiable functions:
\begin{lemma}
\label{lemP}
Let $n\in\mathbb N_0$, and
$\mathcal D\subset \mathbb{R}^d,~~d=2,3.$ If
$g\in C^n(\mathcal D)$, then
\begin{align}
\lVert g^2\rVert_{C^n(\mathcal D)}
\le
2^n\lVert g\rVert_{C^n(\mathcal D)}^2.
\label{Rp1}
\end{align}
\end{lemma}

\begin{proof}
For any multi-index
$\alpha\in\mathbb N_0^d$ with $\lvert\alpha\rvert\le n$, the
higher-order product rule gives
\begin{align}
D^\alpha(g^2)
=
\sum_{\beta\le\alpha}
\binom{\alpha}{\beta}
D^\beta g\,D^{\alpha-\beta}g,
\qquad
\binom{\alpha}{\beta}
=
\prod_{i=1}^d\binom{\alpha_i}{\beta_i}.
\label{Rp4}
\end{align}
Since $\lvert\beta\rvert\le n$ and
$\lvert\alpha-\beta\rvert\le n$, taking the $L^\infty(\mathcal D)$-norm and using the triangle inequality and the definition of the $C^n(\mathcal D)$-norm, it follows from \eqref{Rp4} that
\begin{align}
\lVert D^\alpha(g^2)\rVert_{L^\infty(\mathcal D)}
&\le
\sum_{\beta\le\alpha}
\binom{\alpha}{\beta}
\lVert D^\beta g\rVert_{L^\infty(\mathcal D)}
\lVert D^{\alpha-\beta}g\rVert_{L^\infty(\mathcal D)}
\nonumber\\
&\le
\left(
\sum_{\beta\le\alpha}
\binom{\alpha}{\beta}
\right)
\lVert g\rVert_{C^n(\mathcal D)}^2.
\label{Rp6}
\end{align}
Using the binomial theorem and $\lvert\alpha\rvert\le n$, we obtain
\begin{align}
\sum_{\beta\le\alpha}
\binom{\alpha}{\beta}
&=
\prod_{i=1}^d
\left(
\sum_{\beta_i=0}^{\alpha_i}
\binom{\alpha_i}{\beta_i}
\right)
\nonumber\\
&=
\prod_{i=1}^d 2^{\alpha_i}
=
2^{\lvert\alpha\rvert}
\le
2^n.
\label{Rp11}
\end{align}
Thus, by \eqref{Rp6} and \eqref{Rp11}, we obtain
\begin{align}
\lVert D^\alpha(g^2)\rVert_{L^\infty(\mathcal D)}
\le
2^n\lVert g\rVert_{C^n(\mathcal D)}^2.
\label{Rp12}
\end{align}
Taking the maximum over $\lvert\alpha\rvert\le n$ in
\eqref{Rp12} completes the proof of Lemma~\ref{lemP}.
\medskip

\noindent In order to obtain necessary bounds on the derivatives of the neural network, we need the use of the following lemma:
\begin{lemma}
\label{x2}
\cite[Lemma A.7]{DRyck}
Let $d,m,n\in\mathbb N$, $\Omega_1\subset\mathbb R^d$,
$\Omega_2\subset\mathbb R^m$, $f\in C^n(\Omega_1;\Omega_2)$, and
$g\in C^n(\Omega_2;\mathbb R)$. Then
\begin{align*}
\lVert g\circ f\rVert_{W^{n,\infty}(\Omega_1)}
\le
16\bigl(e^2n^4md^2\bigr)^n
\lVert g\rVert_{W^{n,\infty}(\Omega_2)}
\max_{1\le i\le m}
\lVert(f)_i\rVert_{W^{n,\infty}(\Omega_1)}^n.
\end{align*}
\end{lemma}
\end{proof}

\section{Conflict of Interest}
The authors have  no conflicts of interest to disclose.

\section*{Acknowledgment}
\noindent The author Pelin G. Geredeli would like to thank the National Science Foundation, and acknowledge her partial funding from NSF Grants NSF-DMS-2348312 and NSF-DMS-2606020.
\medskip

\noindent The authors would like to thank Clemson University for providing access to the Palmetto High Performance Computing cluster and the associated computational support.

\bigskip

\end{document}